\documentclass[10pt,reqno]{amsart}
\usepackage{amsmath, latexsym, amsfonts, amssymb,amsthm, amscd,epsfig,enumerate}
\usepackage{subfigure,color, float}

\advance\hoffset -.75cm

\usepackage[usesnames]{xcolor}

\definecolor{refkey}{gray}{.75}
\definecolor{labelkey}{gray}{.75}

\newcommand{\Z}{\mathbb Z}
\newcommand{\N}{\mathbb N}
\newcommand{\Prob}{\mathbb P}

\newcommand{\cH}{\mathcal{H}}

\newcommand{\pr}{\mathbb P}

\newtheorem{teo}{Theorem}[section]

\newtheorem{cor}[teo]{Corollary}
\newtheorem{rem}[teo]{Remark}

\newtheorem{defn}[teo]{Definition}
\newtheorem{exmp}[teo]{Example}

\newtheorem{assump}[teo]{Assumption}

\title%[Recent results on branching random walks]
{A generating function approach \\ to branching random walks}
\author[D.~Bertacchi]{Daniela Bertacchi}
\address{D.~Bertacchi, Dipartimento di Matematica e Applicazioni,
Universit\`a di Milano--Bicocca,
via Cozzi 53, 20125 Milano, Italy.}
\email{daniela.bertacchi\@@unimib.it}

\author[F.~Zucca]{Fabio Zucca}
\address{F.~Zucca, Dipartimento di Matematica,
Politecnico di Milano,
Piazza Leonardo da Vinci 32, 20133 Milano, Italy.}
\email{fabio.zucca\@@polimi.it}

\date{}

\begin{document}

\begin{abstract}
It is well known that the behaviour of a branching process is completely described by the generating function 
of the offspring law and its fixed points.
% Indeed the generating function of a branching process has either one fixed point, which implies almost sure extinction, 
% or two fixed points, which implies survival with positive probability. 
Branching random walks are a natural generalization of branching processes: a branching process can be seen
as a one-dimensional branching random walk.
We define a
multidimensional generating function associated to a given branching random walk.
The present paper investigates the similarities and the differences of the generating functions,
their fixed points and the implications on the underlying stochastic process, between the one-dimensional (branching process)
and the multidimensional case (branching random walk).
In particular, we show that the generating function of a branching random walk can have
uncountably many fixed points and a fixed point may not be an
extinction probability, even in the irreducible case (extinction probabilities are always fixed points). 
Moreover, the generating function might not be a convex function.
We also study how the behaviour of a branching random walk is affected by local modifications of the process.
As a corollary, we describe a general procedure with which we can modify a continuous-time branching random walk which has a weak
phase and turn it into a continuous-time branching random walk which has strong local survival for large or small values
of the parameter and non-strong local survival for intermediate values of the parameter.
\end{abstract}

%\tableofcontents

\maketitle


\section{Introduction}
\label{sec:intro}

A branching process, or Galton-Watson process (see \cite{cf:GW1875}), is a process where a particle dies and 
gives birth to a random number of offspring,
according to a given offspring law $\rho$ ($\rho(n)$ being the probability of having exactly $n$ children).
Different particles breed independently, all according to $\rho$. Unless $\rho(0)=0$, it is not
completely trivial to tell whether the process
survives (with positive probability) or it goes extinct (almost surely). This question can be answered by looking at the fixed points
of the generating function $H(z)=\sum_{n=0}^\infty\rho(n)z^n$, which is defined for $z\in[0,1]$.
There is almost sure extinction if and only if $H$ has only the fixed point $z=1$.
If there are two fixed points, namely $z=1$ and $z=\bar q\in(0,1)$, then there is extinction with probability $\bar q$
and survival with probability $1-\bar q$. If we require that $\rho(1) <1$ then the generating function $H$, being monotone
and convex, has at most two fixed points, so this description settles all the possibilities for the branching process.

A branching random walk (BRW hereafter) is a process where particles are described by their location $x\in X$,
where $X$ is an at most countable set ($X$ is usually interpreted as a spatial variable, but can also be seen
as a ``type'', see for instance \cite{cf:KurtzLyons}). Particles at site $x\in X$ are replaced by a random
number of children, which are placed at various locations on $X$. 
This class of processes (in continuous and
discrete time) has been studied by many authors (see \cite{cf:AthNey, cf:Big1977, cf:Harris63,
cf:Ligg1, cf:MachadoMenshikovPopov, cf:MP00, cf:MP03,  cf:MadrasSchi, cf:MountSchi} just to mention a few);
a survey on the subject can be found in \cite{cf:BZ4}. Note that, in the case of the branching random walk,
there is no upper bound for the number of particles per site. When such an upper bound is fixed, say at most
$m$ particles per site, we get the $m$-type contact process. The branching random walk can be obtained
as the limiting process as $m$ goes to infinity (\cite{cf:BZ3,cf:BZ15,cf:Z1}).

The behaviour of a BRW is in general more complex than the one of a branching process: 
if we start with one particle at a given site $x$, only one of the following holds for the BRW:
(1) it goes almost surely extinct, (2) it survives globally but not locally,
(3) it survives globally and locally but with different probabilities (non-strong local survival),
(4) it survives globally and locally with equal probability (strong local survival).
We stress that there is no strong local survival when either there is non-strong local survival or 
almost sure local extinction. 

Again, some answers can be obtained through the study of the multidimensional generating function $G$, defined
on $[0,1]^X$, associated to the process. It is easy to note that all the extinction probabilities are fixed points of $G$,
therefore if one proves that there is only one fixed point then there is almost sure extinction
({the vector $\mathbf{1}$, defined as $\mathbf{1}(x):=1$ for all $x\in X$, is always a fixed point}).
If there are at least two fixed points 
then there is global survival starting from some vertices and the extinction probability starting from $x$ coincides
with $\bar {\mathbf{q}}(x)$, where $\bar{\mathbf{q}}$ is the smallest fixed point
(see \cite[Corollary 2.2 and Section 3]{cf:BZ2}).
For a long time, it has been believed (see \cite[Theorem 3]{cf:Spataru89}) that, in the irreducible case, no more 
 than two fixed points were possible. This was disproved in \cite{cf:BZ14-SLS}, even though it remains true for irreducible
 BRWs on finite sets (see also \cite[Corollary 3.1]{cf:BZ14-SLS}). In this framework, two questions naturally arise: 
 how many fixed points can the generating function of an irreducible BRW have? 
 At least in the irreducible case, are all fixed points also extinction probabilities?  
Section~\ref{sec:Q&A} provides a negative answer to these questions: Examples~\ref{exmp:reduciblemorefixedpoints}
and \ref{exmp:irreduciblemorefixedpoints} are a reducible and an irreducible BRWs respectively, where 
there are only two extinction probabilities but the set of fixed points is uncountable. 
We also show that the topological properties of the multidimensional $G$ are different from the one-dimensional case: $G$ may not be
convex (Example~\ref{exmp:counterexample1}). Moreover, the set $U_G:=\{\mathbf{z} \in [0,1]^X \colon G(\mathbf{z})\le 
\mathbf{z}\}$ is not {necessarily} convex and its extremal points may be
% contain points which are 
neither fixed points nor extremal points of $[0,1]^X$ (Examples~\ref{exmp:counterexample2} and \ref{exmp:counterexample3}).

Since extinction probabilities are fixed points of $G$, it is clear that if we know that $G$ has only two fixed
points and that the BRW survives locally, then there is strong local survival. Conversely, if there is non-strong
local survival, then there must be at least three fixed points. Somehow related is the question of what happens if 
we modify locally (that is, on a fixed $A\subseteq X$) a given BRW: for instance if the original BRW 
has no strong local survival, what can be said about the modified BRW? 
Theorem~\ref{th:modifiedBRW}
shows that there is global survival and no strong local survival in $A$ in the original BRW if and only if there is global survival and
no strong local survival in $A$ 
in the modified BRW (regardless of the modifications that have been introduced in $A$).
As a corollary we get that if the original BRW dies out locally in $A$ and the modified BRW survives globally,
then almost sure global extinction for the original one is equivalent to strong local survival in $A$ for the modified BRW 
(Corollary~\ref{cor:pureweak-nonstrong-discrete}).
Moreover, for a fixed irreducible BRW, if 
{there is global survival and no strong local survival in some $A \subseteq X$ then there is global survival and no
strong local survival in all finite $B\subseteq X$}.

From these results in discrete time, we are able to prove that, in continuous time, a modification
of the BRW in a finite subset $A$, which lowers the weak critical parameter (something that can usually
be achieved by adding a sufficiently fast reproduction rate at some site), implies that the weak and strong
parameter of the modified BRW coincide. This allows us to describe a general method to produce examples
such as Example~\ref{th:nonstrongandstrong}, where the modified BRW has strong local survival for
some values of the parameter below a threshold and above another threshold, but non-strong local survival
for intermediate values of the parameter (Figures~\ref{fig:t3} and \ref{fig:t3modified}).
This example was originally described in \cite{cf:BZ14-SLS} but appears here with an easier proof and in
a more general framework. Moreover, we prove that in general, a continuous-time BRW which is obtained by
a local modification of another BRW, lowering its weak critical parameter, dies out globally at the weak
critical parameter (which is not always true, see \cite[Example 3]{cf:BZ2}).

Here is the outline of the paper: in Section~\ref{sec:basic} we introduce the terminology, describe
the most common types of BRWs and their features and define the multidimensional generating function associated to a BRW.
Section~\ref{sec:Q&A} is devoted to the questions about the generating function, its fixed points and the extinction
probabilities. In Section~\ref{sec:survivalprob} we address the problem of the possible behaviour of modified BRWs.
Section~\ref{sec:proofs} contains the proofs of the results and the detailed computations for the examples.

\section{Basic definitions and preliminaries}
\label{sec:basic}

The most general example of a BRW lives in discrete time and it can be constructed easily
as a process $\{\eta_n\}_{n \in \N}$
on a set $X$ which is
at most countable, where $\eta_n(x)$ is the number of particles alive
at $x \in X$ at time $n$. 
The dynamics is described as follows:
consider the (countable) measurable space $(S_X,2^{S_X})$
where $S_X:=\{f:X \to \N\colon \sum_yf(y)<\infty\}$ and let 
$\mu=\{\mu_x\}_{x \in X}$ be a family of probability measures
on $(S_X,2^{S_X})$.
A particle of generation $n$ at site $x\in X$ lives one unit of time;
after that, a function $f \in S_X$ is chosen at random according to the law $\mu_x$.
This function describes the number of children and their positions, that is,
the original particle is replaced by $f(y)$ particles at
$y$, for all $y \in X$. The choice of $f$ is independent for all breeding particles.
The BRW is denoted by $(X,\mu)$.
The total number of children associated to $f$ is represented by the
function $\cH:S_X \rightarrow \N$ defined by $\cH(f):=\sum_{y \in X} f(y)$;
the associated law $\rho_x(\cdot):=\mu_x(\cH^{-1}(\cdot))$ is the law of the random number of children
of a particle living at $x$.

Some results rely on the \textit{first-moment matrix}
$M=(m_{xy})_{x,y \in X}$,
where each entry
$m_{xy}:=\sum_{f\in S_X} f(y)\mu_x(f)$ represents
the expected number of children that a particle living
at $x$ sends to $y$
(briefly, the expected number of particles from $x$ to $y$).
 For the sake of simplicity, we require
that $\sup_{x \in X} \sum_{y \in X} m_{xy}<+\infty$.
We denote by $\bar \rho_x:=\sum_{n \ge 0} n \rho_x(n) \equiv \sum_{y \in X} m_{xy}$,
which is the expected number of children of a particle living at $x$.
Given a function $f$ defined on $X$ we denote by $Mf$ the function $Mf(x):=\sum_{y \in X} m_{xy}f(y)$
whenever the right-hand side converges absolutely for all $x$.

If we observe the process at times $i\cdot n$ ($i \in \mathbb{N}$) we obtain a new BRW
whose first-moment matrix is the $n$th power matrix $M^n$ with entries $m^{(n)}_{xy}$.
%its iterates (that is,
We define
\begin{equation}\label{eq:generalgeomparam}
M_s(x,y) := \limsup_{n \to \infty} \sqrt[n]{m_{xy}^{(n)}}, \quad
M_w(x) := \liminf_{n \to \infty} \sqrt[n]{\sum_{y \in X} m_{xy}^{(n)}}, \qquad \forall x,y \in X;
\end{equation}
see \cite{cf:BZ, cf:BZ2} for some explicit computations and 
\cite[Section 3.2]{cf:Z1} for the relation between $M_s(x,x)$
and some generating functions.

It is important to note that, for a generic BRW, 
the locations of the offsprings are not chosen independently but they are assigned by the chosen
function $f\in S_X$.
We denote by  $P$ the \textit{diffusion matrix} with entries $p(x,y)=m_{xy}/\bar \rho_x$.
In particular, if $\bar \rho_x$ does not depend on $x \in X$, we have that $M_w(x)=\bar \rho$ for all $x \in X$ and
$M_s(x,y)=\bar \rho \cdot \limsup_{n \to \infty} \sqrt[n]{p^{(n)}(x,y)}$ (where the $\limsup$ defines
the \textit{spectral radius} of $P$ according to \cite[Chapter I, Section 1.B]{cf:Woess}). 
When the offsprings are dispersed independently, they are placed according to $P$
and the process is called 
\textit{BRW with independent diffusion}:
in this case
\begin{equation}\label{eq:particular1}
\mu_x(f)=\rho_x \left (\sum_y f(y) \right )\frac{(\sum_y f(y))!}{\prod_y f(y)!} \prod_y p(x,y)^{f(y)},
\quad \forall f \in S_X.
\end{equation}

To  a generic discrete-time BRW we associate a graph $(X,E_\mu)$ where $(x,y) \in E_\mu$  
if and only if $m_{xy}>0$.
% We denote by $\mathrm{deg}(x)$
% the degree of a vertex $x$, that
% is, the cardinality of the set $ \mathcal{N}_x:=\{y\in X\colon  (x,y) \in E_\mu%\mu(x,y)>0
% \}$.
We say that there is a path from $x$ to $y$, and we write $x \to y$, if it is
possible to find a finite sequence $\{x_i\}_{i=0}^n$ (where $n \in \N$)
such that $x_0=x$, $x_n=y$ and $(x_i,x_{i+1}) \in E_\mu$
for all $i=0, \ldots, n-1$ (observe that there is always a path of length $0$ from $x$ to itself).
Whenever $x \to y$ and $y \to x$ we write $x \rightleftharpoons y$.
 The equivalence relation $\rightleftharpoons$ induces a partition of $X$: the
class $[x]$ of $x$ is called \textit{irreducible class of $x$}.
It is easy to show that if $x \rightleftharpoons x^\prime$ and $y \rightleftharpoons y^\prime$ then
$M_s(x,y)=M_s(x^\prime,y^\prime)$ and $M_w(x)=M_w(x^\prime)$. Moreover, 
$m^{(n)}_{xx}$ and $M_s(x,x)$ depend only on the entries $(m_{ww^\prime})_{w,w^\prime \in [x]}$. 
If the graph $(X,E_\mu)$ is \textit{connected} (that is, there is only one irreducible class)
then we say that the first-moment matrix $M$ %=(m_{xy})_{x,y \in X}$
is \textit{irreducible},  
otherwise we call it \textit{reducible}; the same notation applies to the BRW.
The irreducibility of $M$ implies that the progeny of any particle can spread to any site of the 
graph. For an irreducible BRW, $M_s(x,y)=M_s$ and $M_w(x)=M_w$ for all $x,y \in X$.

% The BRW $(X,\mu)$ is called \textit{non-oriented} or \textit{symmetric} if $m_{xy}=m_{yx}$ for every $x,y \in X$.
% Note that if $(X,\mu)$ is non-oriented then the graph $(X,E_\mu)$ is non-oriented (that is, $(x,y) \in E_\mu$
% if and only if $(y,x) \in E_\mu$).
% $(X,\mu)$ is called \textit{nonamenable} if and only if
% \[ %\begin{equation}%\label{eq:amenable}
% \inf
% \left \{
% \frac{\sum_{x \in S, y \in S^\complement} m_{xy}}{|S|}\colon  S \subseteq X, |S| < \infty
% \right \}%=:\iota_{(X,\mu)}
% >0,
% \] %\end{equation}
% and it is called \textit{amenable} otherwise.
% 
% The idea behind the definition of nonamenability is that the expected number of children placed outside every finite subset of 
% $X$ is always comparable with the size of the subset itself. This suggests, in principle, that it should be possible 
% for the BRW to survive and, at the same time, to escape from every finite set.
% This is true for a subclass of BRWs but not in general, see Example~\ref{exm:amenable} and the preceding discussion.
% We note that, if $m_{xy}\in\{0, \lambda\}$ (for some fixed $\lambda>0$) then the BRW is nonamenable if and only if the graph $(X, E_\mu)$ is  
% nonamenable according to the usual definition for graphs (see \cite[Chapter II, Section 12.B]{cf:Woess}).

We consider initial configurations with only one particle placed at a fixed site $x$ and we
denote by $\pr^{\delta_{x}}$ the law of the corresponding process. The evolution of the process
with more complex initial conditions can be obtained by superimposition.
% Throughout this paper 
In the following, \textit{wpp}
is shorthand for ``with positive probability'' (although, when talking about survival, 
``wpp'' will be usually tacitly understood).
In order to avoid trivial situations where particles have one offspring almost surely, we assume
henceforth the following.
\begin{assump}\label{assump:1}
For all $x \in X$ there is a vertex $y \rightleftharpoons x$ such that
$\mu_y(f\colon  \sum_{w\colon w \rightleftharpoons y} f(w)=1)<1$,
 that is, in every equivalence class (with respect to $\rightleftharpoons$)
there is at least one vertex where a particle
can have inside the class a number of children different from 1 wpp.
\end{assump}
We now distinguish between the possible behaviours of a BRW.
\begin{defn}\label{def:survival} $\ $
\begin{enumerate}
 \item 
The process \textsl{survives locally wpp} in $A \subseteq X$ starting from $x \in X$ if
%\[
$
{\mathbf{q}}(x,A)
% ={\mathbf{q}}^\mu(x,A)
:=1-\pr^{\delta_x}(\limsup_{n \to \infty} \sum_{y \in A} \eta_n(y)>0)<1.
$
%\]
\item
The process \textsl{survives globally wpp} starting from $x$ if
$
\bar {\mathbf{q}}(x)
% =\bar {\mathbf{q}}^\mu(x)
:={\mathbf{q}}(x,X) 
%\equiv 1- \pr^{\delta_x} \Big (\sum_{w \in X} \eta_n(w)>0, \forall n \in N \Big )
<1$.
\item
There is \textsl{strong local survival wpp} in $A \subseteq X$ starting from $x \in X$
if
$ 
{\mathbf{q}}(x,A)=\bar {\mathbf{q}}(x)<1
$ %\]
and \textsl{non-strong local survival wpp} in $A$ if $\bar {\mathbf{q}}(x)<{\mathbf{q}}(x,A)<1$.
\item
The BRW is in a \textit{pure global survival phase} starting from $x$ if
$ %\[
\bar {\mathbf{q}}(x)<{\mathbf{q}}(x,x)=1
$ %\]
(where we write ${\mathbf{q}}(x,y)$ instead of ${\mathbf{q}}(x, \{y\})$ for all $x,y \in X$).
\end{enumerate}
\end{defn}
\noindent
According to the previous definition,
the probabilities of extinction in $A$ starting from $x$ are denoted by
${\mathbf{q}}(x,A)$, which depend on $\mu$. When we need to stress this dependence, we write 
$ {\mathbf{q}}^\mu(x,A)$.
When $x=y$ we will simply say that local survival occurs ``starting from $x$'' or ``at $x$'. 
When there is no survival wpp, we say that there is extinction
and the fact that extinction occurs %with probability 1 
almost surely
will be tacitly understood.
There are many relations between $\bar {\mathbf{q}}(x)$ and ${\mathbf{q}}(x,y)$ and between ${\mathbf{q}}(w,x)$ and
${\mathbf{q}}(w,y)$ where $x,y, w \in X$ (see for instance Section~\ref{sec:Q&A} or \cite{cf:BZ4, cf:Z1}).

Roughly speaking, strong local survival means that for almost all realizations the process either survives locally
(hence globally) or it goes globally extinct. More precisely,
there is strong survival at $y$ starting from $x$ if and only if the probability
of local survival at $y$ starting from $x$ conditioned on global survival starting from $x$ is $1$.

We want to stress that $\bar {\mathbf{q}}(x)={\mathbf{q}}(x,A)$ if and only if global survival from $x$ is
equivalent to strong local survival at $A$ from $x$. On the other hand $\bar {\mathbf{q}}(x)<{\mathbf{q}}(x,A)$
if and only if there is global survival and no strong local survival at $A$ from $x$ (that is, either local extinction
at $A$ or non-strong local survival at $A$).
Recall that no strong local survival in $A$ from $x$ means that either there is non-strong local survival in $A$ from $x$ 
or there is local extinction in $A$ from $x$.

\subsection{Continuous-time Branching Random Walks}
\label{subsec:continuous}

In continuous time each particle has an exponentially distributed
random lifetime with parameter 1 (death occurs at rate 1). During its lifetime each particle alive at $x$
breeds into $y$ according to the arrival times of its own Poisson process with
parameter $\lambda k_{xy}$ (representing the reproduction rate), 
where $\lambda>0$ and $K=(k_{xy})_{x,y \in X}$ is a nonnegative matrix. 
We denote by  $(X,K)$ the family of continuous-time BRWs (depending on $\lambda>0$).
It is not difficult to see that the introduction
of a nonconstant death rate $\{d(x)\}_{x \in X}$ does not represent a
significant generalization. Indeed,
 one can study %equivalently 
a new BRW with death rate 1 and
reproduction rates $\{\lambda k_{xy}/d(x)\}_{x,y \in X}$; the two processes have the same behaviours in
terms of survival and extinction (\cite[Remark 2.1]{cf:BZ14-SLS}).

To show that the class of continuous-time BRWs is ''contained`` into the class of
discrete-time BRWs, we associate to a continuous-time BRW
a discrete-time counterpart which takes into account all the offsprings 
of a particle before it dies.
Thus, all results in discrete time concerning the probabilities of survival (local, strong local and global)
extend smoothly to the continuous time setting. Conversely, each example in continuous-time induces an analogous
example in discrete-time (just by using the discrete-time counterpart).
In particular, by definition, a continuous-time BRW has some property %$\mathcal P$ 
if and only if its discrete-time counterpart has it.
It is easy to show that $\mu_x$ satisfies %is given by
equation~\eqref{eq:particular1}, where
\begin{equation}\label{eq:counterpart}
\rho_x(i)=\frac{1}{1+\lambda k(x)} \left ( \frac{\lambda k(x)}{1+\lambda k(x)} \right )^i, \qquad
p(x,y)=\frac{k_{xy}}{k(x)}, \qquad k(x):=\sum_{y \in X} k_{xy}.
\end{equation}
% ($k(x):=\sum_{y \in X} k_{xy}$).
% Note that  of a continuous-time BRW is a BRW  and that
% $\rho_x$ depends only on $\lambda k(x)$.
Clearly the discrete-time
counterpart is a BRW with independent diffusion satisfying Assumption~\ref{assump:1}. Moreover
$m_{xy}=\lambda k_{xy}$ and $\bar \rho_x=\lambda k(x)$.

Given $x \in X$, two critical parameters are associated to the
continuous-time BRW: the \textit{global} 
\textit{survival critical parameter} $\lambda_w(x)$ and the  \textit{local} 
 \textit{survival critical parameter} $\lambda_s(x)$ defined as
\[ %begin{equation}\label{eq:criticalparameters}
\begin{split}
\lambda_w(x)&:=\inf \Big \{\lambda>0\colon \,
\pr^{\delta_{x}}\Big (\sum_{w \in X} \eta_t(w)>0, \forall t\Big) >0 \Big \},\\
%\pr^{\delta_{x}}\left(\exists t\colon \eta_t=\mathbf{0}\right)<1\}\\
\lambda_s(x)&:=
\inf\{\lambda>0\colon \,
\pr^{\delta_{x}} \big(\limsup_{t \to \infty} \eta_t(x)>0 \big) >0
\}.
%\pr^{\delta_{x}}\left(\exists \bar t\colon \eta_t(x)=0,\,\forall t\ge\bar t\right)<1\},
  \end{split}
\] 
These values depend only on the irreducible class of $x$; in particular they are constant
if the BRW is irreducible.
The process is called
\textit{globally supercritical}, \textit{critical} or \textit{subcritical}
if $\lambda>\lambda_w$, $\lambda=\lambda_w$ or $\lambda<\lambda_w$;
an analogous definition is given for the local behaviour using $\lambda_s$ instead of $\lambda_w$.
Everytime  the
interval $(\lambda_w(x),\lambda_s(x))$ is not empty 
we say that there exists a \textit{pure global survival phase} starting from $x$.
No reasonable definition of a \textit{strong local survival critical parameter} is possible
(see \cite{cf:BZ14-SLS}).
% as a consequence of Example~\ref{th:nonstrongandstrong}.

Given a continuous-time BRW $(X,K)$, for all  $x,y \in X$, we define %the analogs of $M_s(x,y)$ and $M_w(x)$
\[ %\begin{equation}%\label{eq:geomparam}
K_s(x,y) := \frac{M_s(x,y)}\lambda\equiv\limsup_{n \to \infty} \sqrt[n]{k_{xy}^{(n)}}, \quad
K_w(x) := \frac{M_w(x)}\lambda\equiv\liminf_{n \to \infty} \sqrt[n]{\sum_{y \in X} k_{xy}^{(n)}}, %\qquad \forall ,
\] %\end{equation}
where $M_s(x,y)$ and $M_w(x)$ are the corresponding parameters of the discrete-time counterpart.
$K_s(x,y)$ and $K_w(x)$ depend only on the equivalence classes of $x$ and $y$, hence
if the BRW is irreducible, then they do not depend on $x,y \in X$.

Among continuous-time BRWs, two classes are worth mentioning: \textit{site-breeding} BRWs (where $k(x)$ does not depend on $x \in X$)
and \textit{edge-breeding} BRWs (where $k_{xy} \in \N$,
typically in a multigraph this is the number of edges from $x$ to $y$). 

\subsection{Infinite-dimensional generating function}\label{subsec:genfun}

To the family $\{\mu_x\}_{x \in X}$, we associate a generating function $G:[0,1]^X \to [0,1]^X$,
which can be considered as an infinite dimensional power series. % (see also \cite[Section 3]{cf:BZ2}). 
More precisely,
for all ${\mathbf{z}} \in [0,1]^X$, $G({\mathbf{z}}) \in [0,1]^X$ is defined as the following weighted sum of (finite) products
\[ %\begin{equation}
%\label{eq:genfun}
G({\mathbf{z}}|x):= \sum_{f \in S_X} \mu_x(f) \prod_{y \in X} {\mathbf{z}}(y)^{f(y)},
\] %\end{equation}
where $G({\mathbf{z}}|x)$ is the $x$ coordinate of $G({\mathbf{z}})$.
Note that if we have a realization $\{\eta_n\}_{n \in \mathbb{N}}$ of the BRW then
$G({\mathbf{z}}|x)=\mathbb{E}[\prod_{y \in X} {\mathbf{z}}(y)^{\eta_1(y)} | \eta_0=\delta_x]$.

The family $\{\mu_x\}_{x \in X}$ is uniquely determined by $G$. Indeed, fix a finite $X_0 \subseteq X$ and $x \in X$.
For every $\mathbf{z}$ with support in $X_0$, we have  
$G({\mathbf{z}}|x)= \sum_{f \in S_{X_0}} \mu_x(f) \prod_{y \in X_0} {\mathbf{z}}(y)^{f(y)}$
which can be identified with a power series with several variables (defined on $[0,1]^{X_0}$).
Suppose that we have another generating function $ \overline G$ (associated to  $\{\overline \mu_x\}_{x \in X}$)
such that $G=\overline G$. In particular, $G({\mathbf{z}}|x)=\overline G({\mathbf{z}}|x)$ for every 
$\mathbf{z}$ with support in $X_0$. Thus $\mu_x(f)=\overline \mu_x(f)$ for all $f \in S_{X_0}$.
Since $S_X =\bigcup_{\{X_0 \subseteq X \colon X_0 \textrm{ finite}\}}S_{X_0}$ we have that
$\mu_x(f)=\overline \mu_x(f)$ for
all $f \in S_X$.

Note that $G$ is continuous with respect to the \textit{pointwise convergence topology} of $[0,1]^X$  and nondecreasing
with respect to the usual partial order of $[0,1]^X$ (see \cite[Sections 2 and 3]{cf:BZ2} for further details);
everytime we say that an element of $[0,1]^X$ is the smallest (resp.~largest) among a set of points $\mathcal{A}$, we are also
implying that it is comparable with every element of the specific set $\mathcal{A}$. We stress that $\mathbf{z} < \mathbf{w}$ means
$\mathbf{z}(x) \le \mathbf{w}(x)$ for all $x \in X$ and $\mathbf{z}(x_0) < \mathbf{w}(x_0)$ for some $x_0 \in X$.
Moreover, $G$ represents the 1-step reproductions; we denote by $G^{(n)}$ the generating function
associated to the $n$-step reproductions, which is inductively defined as $G^{(n+1)}({\mathbf{z}})=G^{(n)}(G({\mathbf{z}}))$,
where $G^{(0)}$ is the identity.
Extinction probabilities are fixed points
of $G$ and the smallest fixed point is $\bar {\mathbf{q}}$ (see Section~\ref{sec:Q&A} for details):
more generally, given a solution of $G(\mathbf{z}) \le \mathbf{z}$ then $\mathbf z \ge \bar {\mathbf{q}}$.
% We denote by 
% $F_G:=\{\mathbf{z} \in [0,1]^X \colon G(\mathbf{z})=\mathbf{z}\}$ and $U_G:=\{\mathbf{z} \in [0,1]^X \colon G(\mathbf{z})\le 
% \mathbf{z}\}$.

When $(X,\mu)$ is a BRW with independent diffusion, we can compute explicitly $G$: indeed
% An example where the function $G$ can be explicitly computed : in this case it is not difficult to see that
$G({\mathbf{z}}|x)=\sum_{n \in \N} \rho_x(n) (P{\mathbf{z}}(x))^n$
where $P{\mathbf{z}}(x)=\sum_{y \in X} p(x,y){\mathbf{z}}(y)$.
If, in particular, $\rho_x(n)=\frac{1}{1+\bar \rho_x} (\frac{\bar \rho_x}{1+\bar \rho_x} )^n$ 
(as in the discrete-time counterpart of a continuous-time BRW) then
the previous expression becomes $G({\mathbf{z}}|x)=(1+\bar \rho_x P(\mathbf{1}-\mathbf{z})(x))^{-1}$ %.
%The previous equality can be written 
or, in a more compact way, % as
\begin{equation}\label{eq:Gcontinuous}
 G({\mathbf{z}})= \frac{\mathbf{1}}{\mathbf{1}+M(\mathbf{1}-{\mathbf{z}})}
\end{equation}
where $M$ is the first-moment matrix and $M \mathbf{v}(x)=\bar \rho_x P\mathbf{v}(x)$ 
(by definition of $P$).

\subsection{Projection of BRWs}
\label{subsec:projections}

We introduce the concept of projection of a BRW onto another one
(see also \cite{cf:BZ4, cf:BZ14-SLS} where this property is called
\textit{local isomorphism}).

\begin{defn}
\label{def:locallyisomorphic} 
A BRW $(X, \mu)$ is projected onto a BRW $(Y,\nu)$ if there exists a
surjective map $g:X\to Y$ such that
$
%\begin{equation}
%\label{eq:fgraph}
%\nu_{g(x)}(f)=\mu_x\left(h\colon \forall y\in Y, f(y)=\sum_{z\in g^{-1}(y)}h(z)\right),
%\quad \forall f\in S_Y.
\nu_{g(x)}(\cdot)=
\mu_x\left(\pi_g^{-1}(\cdot)\right)
$, %\end{equation}
where $\pi_g:S_X \rightarrow S_Y$ is defined as $\pi_g(f)(y)=\sum_{z\in g^{-1}(y)}f(z)$ for all $f\in S_X$, $y \in Y$.
\end{defn}
Clearly, if $(X,\mu)$ is projected onto $(Y, \nu)$ then, for all ${\mathbf{z}} \in [0,1]^Y$ and $x \in X$,
\begin{equation} \label{eq:Gfunctions}
G_X({\mathbf{z}} \circ g|x)=G_Y({\mathbf{z}}|g(x)).
\end{equation}
Since $\mu$ is uniquely determined by $G$, %thus  
 equation~\eqref{eq:Gfunctions} holds  if and only if $(X,\mu)$ is projected 
onto $(Y,\nu)$ and $g$ is the map in Definition~\ref{def:locallyisomorphic}.
% To see the ``only if'' part, define $\widehat \nu$ by using equation~\eqref{eq:fgraph} (substitute $\nu$ with $\widehat \nu$),
% then equation~\eqref{eq:Gfunctions} holds with
% $\widehat G$ instead of $G_Y$; thus $\widehat G=G_Y$ and this implies that equation~\eqref{eq:fgraph}
% holds for $\nu$.
The rough idea behind this definition is to assign to every $x \in X$ a label ($g(x)$ drawn from $Y$) in such a way that, 
if $\{\eta_n\}_{n \in \N}$ is a realization of the BRW $(X, \mu)$ then the sum of the particles 
over all vertices with the same label, that is
$\{\pi_g(\eta_n)\}_{n \in \N}$, is a realization of the BRW $(Y, \nu)$.

Note that equation~\eqref{eq:Gfunctions} can be written as $G_X({\mathbf{z}} \circ g)=G_Y({\mathbf{z}})\circ g$ hence
$G^{(n)}_X({\mathbf{z}} \circ g)=G^{(n)}_Y({\mathbf{z}})\circ g$ for all $n \in \N$. As a consequence, for
the global extinction probabilities of these BRWs, we have $\bar{\mathbf{q}}_X=\bar{\mathbf{q}}_Y \circ g$;
indeed $\mathbf{0}_X=\mathbf{0}_Y \circ g$, thus
$\bar{\mathbf{q}}_X=\lim_{n \to \infty}G_X({\mathbf{0}_X})=\lim_{n \to \infty}G_Y({\mathbf{0}_Y}) \circ g=\bar{\mathbf{q}}_Y \circ g$.

A BRW which can be projected onto a BRW defined on a finite set,  
 is called \textit{$\mathcal{F}$-BRW} (see \cite[Section 2.4]{cf:BZ14-SLS}).
To give an explicit example, consider a BRW with independent diffusion on a tree with two alternating degrees:
this can be projected onto a BRW on a set of cardinality 2. Other examples are \textit{quasitransitive BRWs} 
(see \cite[Section 2.4]{cf:BZ14-SLS} for the formal definition) where the action of the group of automorphisms (bijective
maps preserving the reproduction laws) has a finite number of orbits. 
 There are non-quasitransitive BRWs which are $\mathcal{F}$-BRWs (see \cite[Figure 1]{cf:BZ14-SLS}).
More generally, 
let us define the map $\varpi_g:[0,1]^Y \to [0,1]^X$ by $\varpi_g(\mathbf{z})=\mathbf{z}\circ g$;
then $\varpi_g(F_{G_Y}) \subseteq F_{G_X}$, indeed, using equation~\eqref{eq:Gfunctions},
$G_x(\varpi_g(\mathbf{z})|x)=G_X({\mathbf{z}} \circ g|x)=G_Y({\mathbf{z}}|g(x))=\mathbf{z}(g(x))=\varpi_g(\mathbf{z})(x)$.
In particular, the set $F_{G_X}$ is closed under the action of all maps $\varpi_g$ for every projection
$g$ of $(X,\mu)$ onto itself. Moreover, it is easy to show that $\mathbf{q}^X(\cdot,g^{-1}(A))=\varpi(\mathbf{q}^Y(\cdot,A))$
for all $A\subseteq X$.

Another example, is the case of BRWs where the laws of the offspring number $\rho_x=\rho$ is independent
of $x \in X$; we call them \textit{Branching Process}-like BRWs (or
\textit{BP}-like BRWs). In this case the BRW can be projected onto a BRW defined on a singleton $Y:=\{y\}$, where 
the law of the number of children of each particle is $\rho$ and $g(x):=y$ for all $x \in X$
(and this last BRW is actually a branching process). It is worth noting that in this case Assumption~\ref{assump:1}
is simply $\rho(1)<1$.
This kind of BRWs has been studied in \cite{cf:BZ4, cf:BZ14-SLS}
where they are called 
\textit{locally isomorphic to a branching process}.
By using %the methods in \cite{cf:BZ4, cf:BZ14-SLS} 
the equality $\bar{\mathbf{q}}_X=\bar{\mathbf{q}}_Y \circ g$
we have that
 $\bar{\mathbf{q}}$
is a constant vector $c \cdot \mathbf{1}$, where 
$c$ is the smallest fixed point of the function $z \mapsto \sum_{i=0}^\infty \rho(z) z^i$.

\subsection{Conditions for survival/extinction}

We summarize here some conditions for survival and extinction in discrete and continuous time that we need
in the rest of the paper.
For the proofs and further results we refer, for instance, to \cite{cf:BZ, cf:BZ2, cf:BZ4, cf:Z1}.

\begin{teo}\label{th:discretesurv}
Let $(X,\mu)$ be a discrete-time BRW.
\begin{enumerate}
 \item 
There is
local survival starting from $x$ if and only if $M_s(x,x) %\limsup_{n\to\infty}\sqrt[n]{m^{(n)}_{xx}}
>1$.
\item
There is global survival starting from $x$ if and only if there exists
${\mathbf{z}}\in [0,1]^X$, ${\mathbf{z}}(x)<1$
such that $G({\mathbf{z}}|y) = {\mathbf{z}}(y)$, for all $y \in X$
(equivalently, such that $G({\mathbf{z}}|y) \le {\mathbf{z}}(y)$, for all $y \in X$).
\item
If $(X,\mu)$ is an $\mathcal F$-BRW then
there is global survival starting from $x$ if and only if $M_w(x)>1$.
\end{enumerate}
\end{teo}

% We note immediately that l
Local survival depends only on the first-moment matrix while global
survival, except for particular classes as explained in \cite[Section 3.1]{cf:BZ14-SLS}, does not.
Moreover, each solution $\mathbf{z}$ of the inequality in Theorem~\ref{th:discretesurv}(2) satisfies
$\mathbf{z} \ge \bar {\mathbf{q}}$, since the latter is the smallest among such solutions.

For a BRW with independent diffusion, from equation~\eqref{eq:Gcontinuous} and Theorem~\ref{th:discretesurv}(2)
we have that there is
 global survival starting from $x$, if and only if there exists
${\mathbf{v}} \in [0,1]^X$, ${\mathbf{v}}(x)>0$ such that
\begin{equation}\label{eq:ineq}
 M{\mathbf{v}} \ge {\mathbf{v}} /(\mathbf{1}-{\mathbf{v}}), \qquad \text{(equivalently, }
 M{\mathbf{v}} = {\mathbf{v}} /(\mathbf{1}-{\mathbf{v}}) \text{).}
\end{equation}
Remember that, for a continuous-time BRW, $M=\lambda K$. As before, 
each solution $\mathbf{v}$ of the previous inequality satisfies
$\mathbf{v} \ge \mathbf{1}- \bar {\mathbf{q}}$, since the latter is the largest among such solutions.
In the continuous-time case however, global and
local survival are related to the critical values $\lambda_w(x)$ and $\lambda_s(x)$ so it is
useful to be able to give some estimates.

\begin{teo}\label{th:continuoussurv}
Let $(X,K)$ be a continuous-time BRW.
\begin{enumerate}
 \item 
$\lambda_s(x)=1/K_s(x,x)$ and if 
$\lambda=\lambda_s(x)$ then there is local extinction at $x$.
\item
$
 \lambda_w(x) \ge 1/K_w(x).
$
\item If $(X,K)$ is an $\mathcal F$-BRWs then  
$\lambda_w(x)=1/K_w(x)$ and when $\lambda=\lambda_w(x)$ there
is global extinction starting from $x$.
\end{enumerate}
\end{teo}
More conditions can be found for instance in \cite{cf:BZ, cf:BZ2, cf:BZ4}. 
In particular, $\lambda_w$ admits a characterization, in the spirit of equation \eqref{eq:ineq},
in terms of a system of functional inequalities (see \cite[Theorem 4.2]{cf:BZ2}).
Even if there can be global survival
% It is noticeable
% that, unlike local survival, even for irreducible BRWs 
when $\lambda=\lambda_w$ 
(see \cite[Example 3]{cf:BZ2}), this is not true for a continuous-time $\mathcal{F}$-BRW. %can be projected onto a BRW on a finite set, 
Indeed, in this case, $\lambda_w(x)=1/K_s(x,x)$ and there is always global extinction starting from $x$
when $\lambda=\lambda_w(x)$ (see \cite[Theorems 4.7 and 4.8]{cf:BZ2}).

So far all results describe conditions for extinction versus survival, that is, $\mathbf{q}(x,A)=1$ versus $\mathbf{q}(x,A)<1$. 
One could also investigate whether $\bar{\mathbf{q}}(x)=\mathbf{q}(x,A)<1$ or $\bar{\mathbf{q}}(x)<\mathbf{q}(x,A)<1$;
to put it another way, what is the probability of local survival conditioned to global survival? 
Studying strong local survival is more complicated than working on local or global survival. Many properties which
can be easily proven when studying local/global behaviour, do not hold for the strong local one.
For instance, as we already observed, % Example~\ref{th:nonstrongandstrong} shows that, 
even the irreducible case, it is not possible give a reasonable definition of a critical parameter for strong local survival as we did
for local and global survival. Moreover, in the irreducible case, local and global behaviours do not depend on the starting vertex
(or, more generally, on the starting configuration as long as it is finite) but this is not true for strong local behaviour
unless $\rho_x(0)>0$ for all $x \in X$
(see Remark~\ref{rem:irreducible} below and \cite[Example 4.3]{cf:BZ14-SLS}).

Some conditions for strong local survival are achieved by using a generating function approach 
(see  \cite[Section 3.2]{cf:BZ14-SLS}, in particular Theorem~3.4 and Corollaries~3.1 and 3.2)
and they are briefly discussed in Section~\ref{sec:Q&A}.
Among other results available in the literature, it is worth mentioning
a characterization of strong local survival originally proven in \cite[Theorem 2.1]{cf:MenshikovVolkov} and extended
to a generic irreducible BRW in \cite[Theorem 3.5]{cf:BZ14-SLS}.
Results on strong local survival for BRWs in random environment can be found, for instance, in \cite{cf:GMPV09}.

\section{Fixed points and extinction probabilities}\label{sec:Q&A}

Define ${\mathbf{q}}_n(x,A)$
% ={\mathbf{q}}^\mu(x,A)$ 
as the probability of extinction in $A$ no later than the $n$-th
generation starting with one particle at $x$, namely
${\mathbf{q}}_n(x,A)=\Prob^{\delta_x}(\eta_k(y)=0,\, \forall k\ge n,\,\forall y\in A)$. The sequence
$\{{\mathbf{q}}_n(x,A)\}_{n \in \N}$ is nondecreasing and satisfies
\begin{equation}\label{eq:extprobab}
 \begin{cases}
 {\mathbf{q}}_{n}(\cdot,A)=G({\mathbf{q}}_{n-1}(\cdot, A)),& \quad \forall n \ge 1\\
 {\mathbf{q}}_0(x,A)=0, &\quad \forall x \in A.
\end{cases}
\end{equation}
Moreover,  ${\mathbf{q}}_n(x, A)$
converges to ${\mathbf{q}}(x,A)$,
which is the probability of local extinction in $A$
starting with one particle at $x$ (see
Definition~\ref{def:survival}). 
Since $G$ is continuous we have that ${\mathbf{q}}(\cdot,A)=G({\mathbf{q}}(\cdot,
A))$, hence these extinction probabilities are
fixed points of $G$, that is, elements of $F_G:=\{\mathbf{z} \in [0,1]^X \colon G(\mathbf{z})=\mathbf{z}\}$.

Note that ${\mathbf{q}}(\cdot, \emptyset)= \mathbf 1$.
Since $\bar {\mathbf{q}}=\lim_{n \to \infty} G^{(n)}(\mathbf{0})$ we have that
$\bar {\mathbf{q}}$ is the smallest fixed point of $G$ in $[0,1]^X$ (see \cite[Corollary 2.2]{cf:BZ2});
we stress here that $\bar {\mathbf{q}}$ is not only the smallest extinction probability vector, but 
the smallest among all fixed points; hence $\bar {\mathbf{q}}= {\mathbf{1}}$ if and only if $F_G$ is a singleton.
Using the same arguments, one can prove that $\bar {\mathbf{q}}$ is the smallest
fixed point of $G^{(m)}$ for all $m \in \N$.

A meaningful consequence of the convergence ${\mathbf{q}_n}(x,A) \uparrow {\mathbf{q}}(x,A)$ is that, whenever
${\mathbf{q}}(x,A)<1$, the probability of survival 
conditioned on surviving at $A$ up to a time larger than or equal to $n$ converges to $1$, that is, 
$(1-{\mathbf{q}}(x,A))/(1-{\mathbf{q}_n}(x,A)) \uparrow 1$. In the case
$A=X$, ``surviving up to a time larger than or equal to $n$'' is equivalent to 
``surviving up to time $n$''; thus, given that the population survived globally up to a sufficiently large (but finite)
time $n$ then the conditional
probability of survival is arbitrarily close to $1$.

Note that $A \subseteq B \subseteq X$ implies ${\mathbf{q}}(\cdot,A)
\ge {\mathbf{q}}(\cdot,B)
\ge \bar {\mathbf{q}}$. 
From this we can derive trivial implications between local survival or extinction in $A$ and $B$. In particular, strong 
local survival in $A$ from $x$
implies strong local survival in $B$ from $x$; moreover, non-strong local survival in $B$ from $x$ implies either non-strong 
local survival in $A$ from $x$ or local extinction in $A$ from $x$.

Since for all finite $A\subseteq X$ we have ${\mathbf{q}}(x,A) \ge 1-\sum_{y \in A} (1-{\mathbf{q}}(x,y))$
then, for any given finite $A \subseteq X$, ${\mathbf{q}}(x,A)=1$ if and only if ${\mathbf{q}}(x,y)=1$ for all $y \in A$.

If $x \to x^\prime$ and $A \subseteq X$ then ${\mathbf{q}}(x^\prime,A)<1$ implies
${\mathbf{q}}(x,A)<1$; as a consequence,
if $x \rightleftharpoons x^\prime$ then
${\mathbf{q}}(x,A)<1$ if and only if ${\mathbf{q}}(x^\prime,A)<1$. Moreover if  $y \rightleftharpoons y^\prime$ 
we have ${\mathbf{q}}(x,y)={\mathbf{q}}(x,y^\prime)$ for all $x \in X$.
% In the irreducible case $q(x,A)<1$  for some $x \in X$ if and only if
% $q(w,A)<1$  for all $w \in X$; in particular
% $\bar q(x)<1$ for some $x \in X$ if and only if $\bar q(w)<1$ for all $w \in X$.
% Moreover $q(x,A)<1$ for some $x \in X$ and a finite $A \subseteq X$ if and only if
% $q(w,B)<1$ for all $w \in X$ and all finite $B \subseteq X$. Indeed, in the irreducible case,
% one can prove that $q(x,A)=q(x,x)$ for all $x \in X$ and every finite $A \subseteq X$:
% since surviving in a finite subset $A$ is equivalent to surviving in at least
% one of its points, then it is enough to prove it in the case $A:=\{y\}$ for $y \in X$;
% in this case the conclusion follows from a Borel-Cantelli argument.
The main properties in the irreducible case are summarized in the following remark.

\begin{rem}\label{rem:irreducible}
In the irreducible case, for every $x \in X$ and $A \subseteq X$ finite and nonempty,  we have 
${\mathbf{q}}(x,A)={\mathbf{q}}(x,x)$. Thus ${\mathbf{q}}(x,A)={\mathbf{q}}(x,B)$ for every couple $A,B$ of finite, nonempty subsets
of $X$.

If, in addition, $\rho_x(0)>0$ for all $x \in X$, we have that if $\bar {\mathbf{q}}(x)={\mathbf{q}}(x,A)$ for
some $x \in X$ and a finite subset $A \subseteq X$ then $\bar {\mathbf{q}}(y)={\mathbf{q}}(y,B)$ for
all $y \in X$ and all (finite or infinite) subsets $B \subseteq X$ (hence, strong local survival is a common property
of all subsets and all starting vertices, see Theorem~\ref{th:modifiedBRW}).
% Indeed, if $\bar q(x)=1$ then $q(y,B)=1$ for all $y \in X$ and $B \subseteq X$ and there is nothing to prove.
% Suppose that $\bar q(x)=q(x,A)<1$ and $\bar q(y)<q(y,B)$ for some $x,y \in X$ and $A,B \subseteq X$
% finite. By irreducibility $q(x,A)=q(x,x)=q(x,B)$ hence we can assume that $A=B$.
% We know that there is a positive probability that the process, starting from $x$
% has at least one descendant at $y$. There is also a positive probability that all the particles (except one at $y$)
% die out and the progeny of the surviving particle survives globally but not locally in $A$. Thus,
% there is a positive probability, starting from $x$, of surviving globally but not locally in $A$ and this is a contradiction.
% If $B$ is infinite and $z \in B$ then $\bar {\mathbf{q}}(y)={\mathbf{q}}(y,z) \ge {\mathbf{q}}(y,B) \ge \bar {\mathbf{q}}(y)$
%
Clearly, this may not be true in the reducible case. 
Besides, if we drop the assumption $\rho_x(0)>0$ for all $x \in X$, we might actually
have $\bar {\mathbf{q}}(x)={\mathbf{q}}(x,A)<1$ and $\bar {\mathbf{q}}(y)<{\mathbf{q}}(y,A)$ for some $x,y \in X$ and a finite $A \subseteq X$
 even when the BRW is irreducible (see \cite[Example 4.3]{cf:BZ14-SLS}).
Hence, in general, even for irreducible BRWs, 
strong local survival is not a common property of all vertices
 as local and global survival are.
\end{rem}

As we recalled in the introduction, the generating function $G$ of a branching process 
has at most two fixed points in $[0,1]$,  $\bar q$ and $1$.
This is still true for BRWs on finite sets $X$ (see for instance \cite[Corollary 3.1]{cf:BZ14-SLS} or
the proof of \cite[Theorem 3]{cf:Spataru89} which is incorrect in the infinite case, but correct in the finite one).
Moreover, for a branching process, $G$ is strictly convex
and $U_G$ is closed, compact and convex (recall that $U_G$ was defined in Section~\ref{sec:intro}
as $\{\mathbf{z} \in [0,1]^X \colon G(\mathbf{z})\le \mathbf{z}\}$).
Let us denote by $E_G$
the set of extinction probabilities: $E_G:=\{\mathbf{q}(\cdot,A)\colon A \subseteq X\}$.
For a branching process it is true that the extremal points of $U_G$ are the fixed points $\bar q$ and $1$
(where $\bar q$ may coincide with 1) and
all fixed points are extinction probabilities: in short, $\textrm{ext}(U_G)=F_G$
and $F_G=E_G$.

Some of these properties still hold in the general case, others do not, even when $X$ is finite.
It is clear that $F_G$ and 
$U_G$  are always closed and compact sets
 (with respect to the product topology of $[0,1]^X$), since they are closed subsets of the compact topological space $[0,1]^X$. 
We provide some counterexamples and conjectures on the other properties in the following sections.

\subsection{Convexity of $G$ and $U_G$ and extremal points}\label{subsec:convexity}

Given any $\mathbf{w} \le \mathbf{z} \in [0,1]^X$ it is true that $t \mapsto G(\mathbf{w} +t(\mathbf{z}-\mathbf{w}))$ is convex,
nevertheless $G$ is not always a convex function, even when $X$ is finite, as the following example shows.

\begin{exmp}\label{exmp:counterexample1}
Let $X=\{1,2\}$ and $\mu_1=\delta_{(1,1)}$, $\mu_2=\frac12 \delta_{(0,0)} +\frac12 \delta_{(1,0)}$.
Roughly speaking, every particle at 1 has one child at $1$ and one at $2$ almost surely, while every particle at 2 has one child at 1 with probability $1/2$
and no children with probability $1/2$.
The generating function is
\[
 G(x,y)= 
\begin{pmatrix}
 xy \\
(1+x)/2
\end{pmatrix}
\]
which is not convex. Nevertheless $U_G=\{(x,y)\in [0,1]^2\colon 2y \ge x+1\}$ is convex and $F_G=\{(0,1/2), \, (1,1)\}$.
Clearly $\textrm{ext}(U_G)=F_G \cup \{(0,1)\}$.
\end{exmp}
The following two examples show that not only $U_G$ is not necessarily convex, but also its extremal points
may not be elements of $F_G \cup \{0,1\}^X$.% (\{0,1\}^X\setminus\{\mathbf{0}\})$.

% Inspired by the previous example one can wonder if it is always true, at least, that $U_G$ is convex and that
% it is a subset of $F_G \cup (\{0,1\}^X\setminus\{\mathbf{0}\})$. Again, both these conjectures are wrong
% in the general case as the following example shows.
\begin{exmp}\label{exmp:counterexample2}
Let $X=\{1,2\}$ and consider
% To show that $U_G$ is not a subset of $F_G \cup (\{0,1\}^X\setminus\{\mathbf{0}\})$ take 
\[
 G(x,y)= 
\begin{pmatrix}
 (1+3y^2)/4\\
(1+3x^2)/4
\end{pmatrix}
\]
which corresponds to the process where each particle has no children with probability $1/4$ and
2 children on the other vertex with probability $3/4$. In this case $F_G$ contains two vertices on the bisector (one of them is $(1,1)$ of course)
while $U_G$ is the intersection of $(1+3y^2)/4\le x$ and $(1+3x^2)/4 \le y$ and the set of its extremal points is the whole boundary.
\end{exmp}
\begin{exmp}\label{exmp:counterexample3}
% Finally to show that $U_G$ is not convex t
Take $X:=\{1,2,3\}$, $\mu_1=\delta_{(0,1,1)}$, $\mu_2=\delta_{(1,2,1)}$ and
$\mu_3=\delta_{(1,1,0)}$. Roughly speaking every particle at $j$ has two children: one in each point different from $j$.
The generating function is
\[
 G(x_1,x_2,x_3)= 
\begin{pmatrix}
 x_2x_3 \\
x_1x_3 \\
x_1x_2
\end{pmatrix}.
\]
According to~\cite[Corollary 3.1]{cf:BZ4} for a finite-dimensional, irreducible BRW
there are at most two solutions of  $G(\mathbf{z}) \ge \mathbf{z}$ when $\mathbf{z} \ge \bar {\mathbf{q}}$, 
that is, $\bar {\mathbf{q}}$ and $\mathbf{1}$ (in this case the 
vertices $(1,1,1)$ and $(0,0,0)$, which are the only fixed points). It is easy to see that
$(1/2,1/2,1)$ and $(1/2,1,1/2)$ are in $U_G$. The line connecting these points can be parametrized as 
$\mathbf{z}(t):=(1/2, 1/2+t/2, 1-t/2)$, $t \in [0,1]$ %; clearly $z(0), z(1) \in U_G$, however
and $\mathbf{z}(t) \not \in U_G$ for all $t \in (0,1)$
(since $G(\mathbf{z}(t)) \not \leq \mathbf{z}(t)$ for all $t \in (0,1)$). Figures~\ref{fig:UG_top} and \ref{fig:UG_bottom} show the shape
of $U_G$ as seen from the top (vertex $(1,1,1)$) and from the bottom (vertex $(0,0,0)$).

 \begin{figure}[H]
 \centering
 \begin{minipage}{0.45\textwidth}
 \centering
  \includegraphics[width=0.65\textwidth]{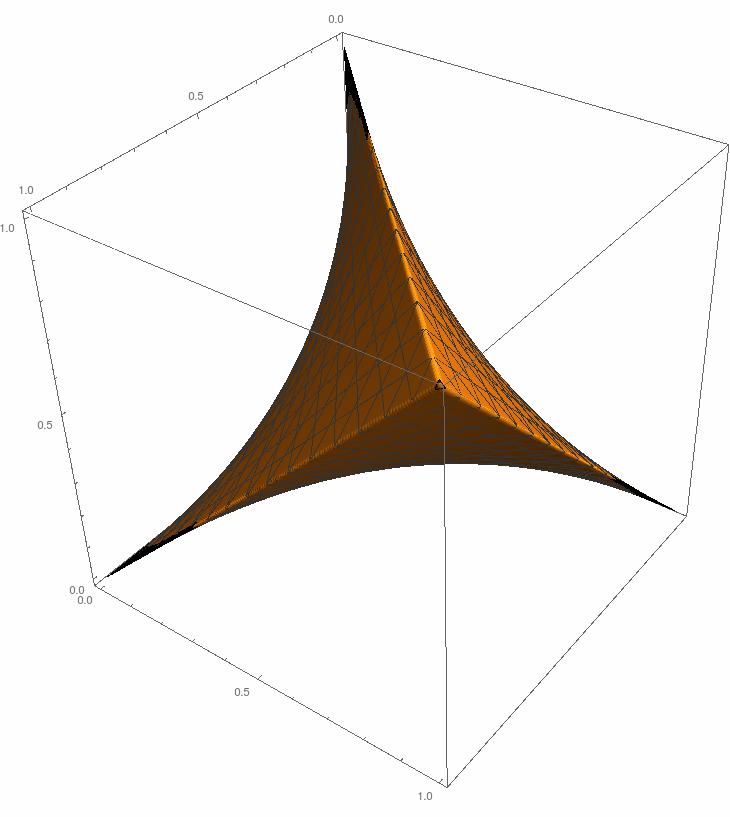}
  \caption{$U_G$ from the top.}\label{fig:UG_top}
 \end{minipage}
\begin{minipage}{0.45\textwidth}
\centering
 \includegraphics[width=0.65\textwidth]{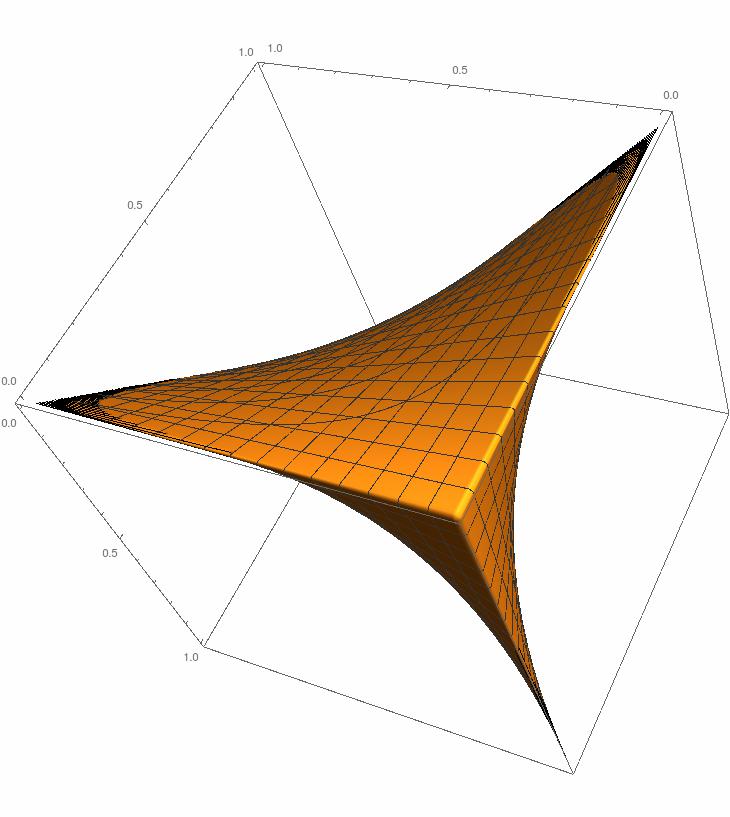}
  \caption{$U_G$ from the bottom.}\label{fig:UG_bottom}
\end{minipage}
\end{figure}
\end{exmp}

\subsection{How many fixed points does $G$ have?}\label{subsec:numberfixedpoints}

If a BRW is reducible, then there can be an infinite (even uncountable) number of fixed points. Consider 
this completely disconnected BRW: let $\{\rho_x\}_{x \in X}$ be an infinite collection of reproduction laws
of supercritical branching processes, define the expected number of children $m_x:=\sum_{n \in \N} n \rho_x(n)>1$
and denote by $c_x<1$ the extinction probability of the $x$th branching process.
Clearly $G(\mathbf{z}|x)=\sum_{n \in \N} \rho_x(n) \mathbf{z}(x)^n$ and 
$F_G=\prod_{x \in X} \{c_x,1\}$, which is uncountable. Moreover, every fixed point is an extinction
probability, since for every $\mathbf{z}\in F_G$, $\mathbf{z}=\mathbf{q}(\cdot, A)$,
where $A:=\{x \in X \colon \mathbf{z}(x)<1\}$. 

Let us discuss the nontrivial case of an irreducible BRW. 
% On the other hand, t
The generating function of an irreducible BRW has at most two fixed points, namely $\bar{\mathbf{q}}$ and $\mathbf{1}$,
when $X$ is finite.
Since $E_G\subseteq F_G$, in order to find examples where $|F_G|\ge3$, it suffices to find cases with $|E_G|\ge3$.
In particular, a BRW with non-strong local survival would do.
In \cite{cf:BZ14-SLS} two such examples were provided: \cite[Examples 4.4 and 4.5]{cf:BZ14-SLS} are irreducible BP-like BRWs with independent diffusion
and non-strong local survival, thus with three different extinction probabilities.

% For many years it has been believed 
% (see \cite[Theorem 3 and Corollary 4]{cf:Spataru89})  that, 
% for an irreducible BRW, $F_G=\{\bar{\mathbf{q}},\mathbf{1}\}$.
% % the generating function of an
% % irreducible BRW could have at most two fixed points, namely $\bar{\mathbf{q}}$ and $\mathbf{1}$.
% Recently, some counterexamples have been constructed (Example~\ref{th:nonstrongandstrong}, \cite[Examples 4.4 and 4.5]{cf:BZ14-SLS})
% where $\{\bar{\mathbf{q}},\mathbf{1}\} \not = E_G$
% and an error in the proof of  \cite[Theorem 3]{cf:Spataru89} has been found (see \cite[Remark 4.1]{cf:BZ14-SLS}).
% Note that \cite[Examples 4.4 and 4.5]{cf:BZ14-SLS} are irreducible BP-like BRWs with independent diffusion
% and non-strong local survival.

It is worth mentioning that
in the case of irreducible, 
quasitransitive BRWs, 
$\{\bar{\mathbf{q}},\mathbf{1}\}  = E_G$
% cannot have extinction probabilities different from $\bar{\mathbf{q}}$ and $\mathbf{1}$:
% if there is 
(local survival starting from some $x \in X$ implies % then there is 
strong local survival starting from all
$x \in X$). Thus $|E_G|=2$ for irreducible, quasitransitive BRWs.
The aforementioned examples in \cite{cf:BZ14-SLS} show that
$\{\bar{\mathbf{q}},\mathbf{1}\} \not = E_G$
% the existence of extinction probabilities different from $\bar{\mathbf{q}}$ and $\mathbf{1}$ 
(thus, non-strong local survival) 
is possible in the case of an irreducible
$\mathcal{F}$-BRW.
We recall that by \cite[Theorem 3.4]{cf:BZ14-SLS}, for an $\mathcal{F}$-BRW, every fixed point $\mathbf{z}$ different from
$\bar{\mathbf{q}}$ satisfies $\sup_{x \in X} \mathbf{z}(x)=1$. In particular, if the BRW is irreducible either
$\mathbf{q}(x,x)=\bar{\mathbf{q}}(x)$ for all $x \in X$ or $\sup_{x \in X} \mathbf{q}(x,x)=1$. 

% The results in \cite[Examples 4.4 and 4.5]{cf:BZ14-SLS}) show that
% $\{\bar{\mathbf{q}},\mathbf{1}\} \not = E_G$
% (thus, non-strong local survival) 
% is possible in the case of an irreducible
% $\mathcal{F}$-BRW. On the other hand, by \cite[Corollary 3.2]{cf:BZ14-SLS} we know that, in the case of irreducible, 
% quasitransitive BRWs, 
% $\{\bar{\mathbf{q}},\mathbf{1}\}  = E_G$:
% local survival starting from some $x \in X$ implies % then there is 
% strong local survival starting from all
% $x \in X$. Nevertheless the above mentioned corollary does not tell anything about the existence of other fixed points.

These remarks do not settle the question of the possible cardinalities of $F_G$, even in the quasitransitive case,
since, as we show in the following section, $F_G$ can be much larger than $E_G$.
Indeed, Example~\ref{exmp:irreduciblemorefixedpoints} proves that, even for an $\mathcal{F}$-BRW,
there may be an uncountable number of fixed points.
It is an open question whether this also holds 
% $\{\bar{\mathbf{q}},\mathbf{1}\} \not = F_G$
for some irreducible, quasitransitive BRW:
% (if there are other fixed points, they are not extinction probabilities):
we conjecture that the
answer is positive (see
Remark~\ref{rem:quasitransitive}).

\subsection{Is every fixed point an extinction probability?}\label{subsec:natureoffixedpoints}

% Strictly related to the previous question is the following one: we know that every extinction probability
% $\mathbf{q}(\cdot, A)$ is a fixed point of $G$, is it true that every fixed point $\mathbf{z}$ equals
% $\mathbf{q}(\cdot, A)$ for some $A \subseteq X$ (that is, $E_G=F_G$)?

The answer is negative.
% Let us consider first the reducible case. If we take the completely disconnected
% BRW mentioned at the beginning of Section~\ref{subsec:numberfixedpoints}, the set of fixed points $F_G=\prod_{x \in X} \{c_x,1\}$ 
% is uncountable,
% nevertheless every $\mathbf{z}\in F_G$ is equal to the extinction probability $\mathbf{q}(\cdot, A)$ where $A:=\{x \in X \colon \mathbf{z}(x)<1\}$. 
% But this is not always the case as the following examples show: for simplicity, 
We start with a reducible example and then we move to an irreducible example.

\begin{exmp}\label{exmp:reduciblemorefixedpoints}
Consider a BRW on $\N$ where every particle at $n$ has two children at $n+1$ with probability $p$
and no children with probability $1-p$ ($p>1/2$ to make it supercritical).
This is a BP-like BRW; 
easy computations (see \cite[Proposition 4.33]{cf:BZ13}) show that 
$G(\mathbf{z}|n)=p \mathbf{z}(n+1))2+1-p$ and $\bar{\mathbf{q}}(x)=(1-p)/p$ for every $x \in \N$. 
Moreover, due to the right drift,
$\mathbf{q}(\cdot,A)=\mathbf{1}$ if $A$ is finite and $\mathbf{q}(\cdot,A)=\bar{\mathbf{q}}$ if $A$ is infinite.
Every fixed point must satisfy $\bar{\mathbf{q}} \le \mathbf{z} \le \mathbf{1}$, thus $\mathbf{z}(0) \in [(1-p)/p,1]$. Clearly if 
$\mathbf{z}(0) =(1-p)/p$ (resp.~$\mathbf{z}(0) =1$) we have $\mathbf{z}=(1-p)/p \cdot \mathbf{1}$ (resp.~$\mathbf{z}=\mathbf{1}$).
Fix $(1-p)/p<\mathbf{z}(0)<1$; the equation $G(\mathbf{z})=\mathbf{z}$ is equivalent to the recursive relation $\mathbf{z}(n+1)=
\sqrt{(\mathbf{z}(n)-(1-p))/p}$. This defines a unique sequence $\mathbf{z}$ which is a fixed point. Indeed
$(1-p)/p<\mathbf{z}(0)<1$ and, by induction, if $(1-p)/p<\mathbf{z}(n)<1$ then $(1-p)^2/p^2<(\mathbf{z}(n)-(1-p))/p<1$, 
thus $(1-p)/p<\mathbf{z}(n+1)<1$.
Obviously, all fixed points can be obtained by means of this procedure, hence the set $F_G$ is uncountable while there are just
two extinction probabilities.
\end{exmp}

\begin{exmp}\label{exmp:irreduciblemorefixedpoints}
% Following the previous, c
Consider the BRW on $\N$ where every particle at $n$ has two children at $n+1$ with probability $p-\varepsilon$,
one child at $\max(0,n-1)$ with probability $\varepsilon$
and no children with probability $1-p$. We require that $2p-\varepsilon>1$ for global survival,  $\varepsilon>0$ for irriducibility,
$p < 1/\sqrt{2}$ and $\varepsilon(p-\varepsilon) \le 1/8$ for technical reasons (take for instance $p=2/3$ and $\varepsilon \le 2/9$).
\begin{figure}[H]
 \centering
 \begin{minipage}{0.9\textwidth}
 \centering
  \includegraphics[width=0.75\textwidth]{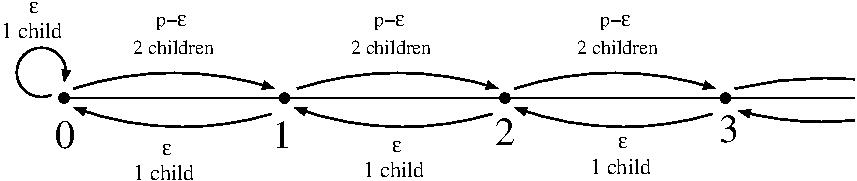}
  \caption{\scriptsize{The irreducible BRW of Example~\ref{exmp:irreduciblemorefixedpoints}.}}\label{fig:exmpirreducible}
 \end{minipage}
\end{figure}
This is an irreducible BP-like BRW (see Figure~\ref{fig:exmpirreducible});
according to Theorem~\ref{th:discretesurv}, global and local survival depend on $M_w$ and $M_s$.
To compute these parameters we refer to (\cite{cf:BZ, cf:BZ2} and \cite[Section 4.6]{cf:BZ13}).
In particular, $M_w$ is the expected number of children $2p-\varepsilon$.
%that we assumed to be strictly larger than one.
 Moreover we have 
 (see \cite[Proposition 4.33]{cf:BZ13})  
$G(\mathbf{z}|n)=(p-\varepsilon) \mathbf{z}(n+1)^2 +\varepsilon\mathbf{z}(\max(0,n-1))+1-p$ and 
$\bar{\mathbf{q}}=(1-p)/(p-\varepsilon) \cdot \mathbf{1}$. Besides,
since $\varepsilon(p-\varepsilon) \le 1/8$ we have local extinction, that is, $M_s\le 1$
(for all the details,  see Section~\ref{sec:proofs} ).
Hence for this BRW there is global survival but local extinction;
thus 
$\mathbf{q}(\cdot,A)=\mathbf{1}$ if $A$ is finite and, since the BRW must drift to the right in
order to survive, $\mathbf{q}(\cdot,A)=\bar{\mathbf{q}}$ if $A$ is infinite.

We prove (see Section~\ref{sec:proofs}) that the equation $G(\mathbf{z})=\mathbf{z}$,
which is equivalent to the recursive equation 
$\mathbf{z}(n+1)=h\big (\mathbf{z}(n),\mathbf{z}(\max(0,n-1)) \big )$ where
$h(x,y):=\sqrt{\big (x - y \varepsilon-(1-p)\big )/(p-\varepsilon)}$,
%$\mathbf{z}(n+1)=\sqrt{\big (\mathbf{z}(n) -\mathbf{z}(\max(0,n-1))\varepsilon-(1-p)\big )/(p-\varepsilon)}$,
defines a unique fixed point for every $\mathbf{z}(0) \in [(1-p)/(p-\varepsilon),1]$ (and every fixed point can be obtained
this way). Thus the set of fixed points is uncountable but there are just two extinction probabilities. 
%We could generalize this example by taking, instead of $2/3$, a generic $p \in (1/2, 1]$ such that $2p-\varepsilon>1$.
\end{exmp}

We conjecture that the previous example extends to quasitransitive BRWs as the following remark suggests.

\begin{rem}\label{rem:quasitransitive}
 Consider the BRW on $\Z$ where every particle at $n$ has two children at $n+1$ with probability $p-\varepsilon$ (such that
$2p-\varepsilon>1$), one child at $n-1$ with probability $\varepsilon$
and no children with probability $1-p$: due to global survival, local extinction and the right drift we have 
just two extinction probabilities, namely
$\mathbf{q}(\cdot,A)=\mathbf{1}$ if $\sup A$ is finite and $\mathbf{q}(\cdot,A)=\bar{\mathbf{q}}$ if $\sup A$ is infinite. 

Suppose that $p$ and $\varepsilon$ satisfy the assumptions   of Example~\ref{exmp:irreduciblemorefixedpoints};
in order to find an uncountable set of fixed points we can proceed as follows.
Any fixed point $\mathbf{z}$ of Example~\ref{exmp:irreduciblemorefixedpoints}, outside $\bar{\mathbf{q}}$ and
$\mathbf{1}$, is a strictly increasing sequence $\{\mathbf{z}(n)\}_{n \in \N}$ converging to $1$.
The function $\phi_n$ mapping $\mathbf{z}(0)$ to $\mathbf{z}(n)$ is continuous, strictly increasing and
maps $(1-p)/(p-\varepsilon)$ and $1$ into themselves; thus $\phi_n$ is an invertible map from
$[ (1-p)/(p-\varepsilon),1 ]$ into itself.
More precisely, $\phi_n$ can be obtained recursively as
% \[
%  \begin{cases}
% \phi_0(x):=x \\
% \phi_1(x):=\sqrt{\big (x(1-\varepsilon)-(1-p)\big )/(p-\varepsilon)} \\
%  \phi_{n+1}(x)=\sqrt{\big (\phi_{n}(x) -\phi_{n-1}(x)\varepsilon-(1-p)\big )/(p-\varepsilon)}.
%  \end{cases}
% \]
\[
 \begin{cases}
\phi_0(x):=x \\
\phi_1(x):=h(x,x) \\
 \phi_{n+1}(x)=h\big (\phi_{n}(x),\phi_{n-1}(x)\big )
 \end{cases}
\]
where $h(x,y):=\sqrt{\big (x - y \varepsilon-(1-p)\big )/(p-\varepsilon)}$ as in Example~\ref{exmp:irreduciblemorefixedpoints}.
Moreover $\{\phi_n(x)\}_{n \in \N}$ is strictly increasing
for all $x \in \big ( (1-p)/(p-\varepsilon),1 \big )$ and constant for all 
$x \in \{ (1-p)/(p-\varepsilon),1 \}$.
Fix $\alpha \in \big ( (1-p)/(p-\varepsilon),1 \big )$ and define
$\mathbf{z}^{(n)} \in [0,1]^\Z$ as
\[
 \mathbf{z}^{(n)}(i):=
 \begin{cases}
 \phi_{n+i}(\phi_n^{-1}(\alpha)) & \textrm{if } i \ge -n, \\
 0 & \textrm{if } i < -n.
 \end{cases}
\]
This is a left-translation of the fixed points of the previous example such that $\mathbf{z}^{(n)}(0)=\alpha$ for every $n \in \N$.
We conjecture that the sequence $\{\mathbf{z}^{(n)}\}_{n \in \N}$ converges (pointwise) to
some $\tilde{\mathbf{z}} \in \big ((1-p)/(p-\varepsilon),1 \big )^\Z$; more precisely we conjecture that 
$\{\mathbf{z}^{(n)}(i)\}_{n \in \N}$ is strictly increasing (resp.~decreasing) when $i$ is positive (resp.~negative). 
If this holds, due to the continuity
of the map $(x,y) \mapsto (p-\varepsilon)x^2 +\varepsilon y +1-p$, then
$\tilde{\mathbf{z}}(n)=(p-\varepsilon)\tilde{\mathbf{z}}(n+1)^2 +\varepsilon \tilde{\mathbf{z}}(n-1)  +1-p$
for every $i \in \Z$; whence, $\tilde{\mathbf{z}}$ is a (non constant) fixed point for the generating function
of the quasitransitive BRW described above.
\end{rem}

% We would like to prove now that $\textrm{ext}(U_G)=F_G \cup (\{0,1\}^X\setminus\{\mathbf{0}\})$ or, at least, 
% that $\textrm{ext}(U_G) \subseteq F_G \cup (\{0,1\}^X\setminus\{\mathbf{0}\})$.

%\subsection{Open question}\label{subsec:open}

%  *************************
% 
% We prove, by means of some examples, that in general $\{\bar{\mathbf{q}}, \mathbf{1}\} \subsetneq E_G \subsetneq F_G$,
% where $\subsetneq$ means there are cases where the inclusion is proper and cases where the equality holds
% (see Table~\ref{eq:table}).
% 
%  
%  *************************

Let us summarize: we proved that, in the irreducible case, 
\begin{equation*}
%\label{eq:table}
 \begin{split}
X \textrm{ finite } &\Longrightarrow \{\bar{\mathbf{q}}, \mathbf{1}\}=E_G=F_G \qquad \qquad
\scriptstyle{\textrm{\cite[Corollary 3.1]{cf:BZ14-SLS}}}\\
X \textrm{ infinite, }(X,\mu) \textrm{ quasitransitive } &\Longrightarrow \{\bar{\mathbf{q}}, \mathbf{1}\}=E_G  \, 
(\subsetneq?)\, F_G \qquad \quad
\scriptstyle{\textrm{\cite[Corollary 3.2]{cf:BZ14-SLS}}}\\
 X \textrm{ infinite, }(X,\mu)\, \mathcal{F}\textrm{-BRW } &\Longrightarrow \{\bar{\mathbf{q}}, \mathbf{1}\}\subsetneq  
 E_G\subsetneq F_G\qquad \qquad
 \scriptstyle{\textrm{\cite[Examples 4.4 and 4.5]{cf:BZ14-SLS}, Example~\ref{exmp:irreduciblemorefixedpoints}}},\\
 \end{split}
\end{equation*}
where $\subsetneq$ means there are cases where the inclusion is proper and cases where the equality holds.
We point out here that the proper inclusion $\{\bar{\mathbf{q}}, \mathbf{1}\} \neq E_G$ is equivalent to non-strong local
survival (for some set $A$ starting from some vertex $x$), while $\{\bar{\mathbf{q}}, \mathbf{1}\} \neq
 F_G$ tells us nothing about strong local survival.
We believe that following the ideas of Remark~\ref{rem:quasitransitive}
one could obtain an example where $E_G  \neq F_G$ for a quasitransitive BRW (hence $E_G  \subsetneq F_G$)
but this exceeds the purpose of this paper.

\section{Strong local survival and local modifications}
\label{sec:survivalprob}

We recall here the following theorem, (it is essentially \cite[Theorem 3.3]{cf:BZ14-SLS}).
In the case of global survival, it gives equivalent conditions for strong local survival
in terms of extinction probabilities .
\begin{teo}\label{th:strongconditioned}
For every nonempty subset $A \subseteq X$, the following assertions are equivalent.
\begin{enumerate}[(1)]
\item  ${\mathbf{q}}(x, A) = \bar {\mathbf{q}}(x)$, for all $x \in X$;
\item ${\mathbf{q}}_0(x,A) \le \bar {\mathbf{q}}(x)$, for all $x \in X$;
% \item  the probability of visiting $A$ at least once starting from $x$ is larger than or equal to the probability of
% global survival starting from $x$, for all $x \in X$:
\item for all $x \in X$, either $\bar {\mathbf{q}}(x)=1$ or
the probability of visiting $A$ at least once starting from $x$ conditioned on global survival starting from $x$ is $1$;
\item  for all $x \in X$, either $\bar {\mathbf{q}}(x)=1$ or
the probability of local survival in $A$ starting from $x$ conditioned on global survival starting from $x$ is $1$
(strong local survival in $A$ starting from $x$).
\item For all $x \in X$ the probability of surviving globally starting from $x$ without ever visiting $A$ is $0$.
\end{enumerate}
\end{teo}
This theorem implies that 
if there exists $x \in X$ such that ${\mathbf{q}}(x,A)>\bar {\mathbf{q}}(x)$ 
(that is, there is a positive probability of global survival and
local extinction in $A$ starting from $x$) then there exists $y \in X$ such that
${\mathbf{q}}_0(y,A)>\bar {\mathbf{q}}(y)$ (which implies that 
there is a positive probability that the BRW survives globally starting from $y$ without
ever visiting $A$, clearly $y \not \in A$). Note that, ${\mathbf{q}}_0(x,A)>\bar {\mathbf{q}}(x)$ implies ${\mathbf{q}}(x,A)>\bar {\mathbf{q}}(x)$ but 
the converse is not true.
Hence we have the following dichotomy: for every fixed nonempty $A$, either ${\mathbf{q}}(\cdot,A)=\bar {\mathbf{q}}(\cdot)$ or 
there is $x \in X\setminus A$ such that there is a positive probability of global survival starting from $x$ without ever visiting $A$.

We note that there is no \textit{a priori} order between the events $A_0:=$``never visit $A$'' and $GE:=$``global 
extinction''. Nevertheless, Theorem~\ref{th:strongconditioned} tells us that if $\mathbf{q}_0(\cdot,A) \le \bar{\mathbf{q}}(\cdot)$
then $\pr^x(A_0 \setminus GE)=0$ for all $x \in X$ (the converse is trivial).

From Theorem~\ref{th:strongconditioned}, which is stated for a single BRW, we derive Theorem~\ref{th:modifiedBRW} and
its Corollaries~\ref{cor:pureweak-nonstrong-discrete} and \ref{cor:pureweak-nonstrong} which give us information about the behaviour
of a BRW after some modifications.

\begin{teo}\label{th:modifiedBRW}
Consider two BRWs  $(X,\mu)$ and $(X,\nu)$. Suppose that  $A \subseteq X$ is a nonempty set
 such that $\mu_x=\nu_x$ for all $x \not \in A$.
 \begin{enumerate}
  \item If we denote by ${\mathbf{q}}^\mu$ and ${\mathbf{q}}^\nu$ the extinction probabilities related to
$(X,\mu)$ and $(X,\nu)$ respectively then we have that
%\begin{enumerate}
%  \item 
% If ${\mathbf{q}}(x,A)>{\mathbf{q}}(x,X)$ then there exists $y \not \in A$ such that  ${\mathbf{q}}(y,A)>{\mathbf{q}}(y,X)$.
% \item If $(X,\nu)$ is a BRW such that $\mu_x=\nu_x$ for all $x \not \in A$ then
${\mathbf{q}}_0^\mu(x,A)={\mathbf{q}}_0^\nu(x,A)$ for all $x \in X$ and %$x \not \in A$. Moreover
\[
 {\mathbf{q}}^\mu(\cdot,A)=\bar {\mathbf{q}}^\mu(\cdot) \Longleftrightarrow
 {\mathbf{q}}^\nu(\cdot,A)=\bar {\mathbf{q}}^\nu(\cdot).
\]
\item If $(X,\mu)$ is irreducible and $B, C \subseteq X$ are two nonempty sets such that $B$ is finite then 
\[
\begin{split}
 {\mathbf{q}}^\mu(\cdot,B)=\bar {\mathbf{q}}^\mu(\cdot) &\Longrightarrow
 {\mathbf{q}}^\mu(\cdot,C)=\bar {\mathbf{q}}^\mu(\cdot). \\
%  \textrm{(resp.~}{\mathbf{q}}^\nu(\cdot,A)=\bar {\mathbf{q}}^\nu(\cdot) \Longleftrightarrow
%  {\mathbf{q}}^\nu(\cdot,B)=\bar {\mathbf{q}}^\nu(\cdot)\ \textrm{)}. 
\end{split}
\]
% Moreover, each of the previous equalities implies
% to 
% ${\mathbf{q}}^\mu(\cdot,B)=\bar {\mathbf{q}}^\mu(\cdot)$ %(resp.~${\mathbf{q}}^\nu(\cdot,B)=\bar {\mathbf{q}}^\nu(\cdot)$)
% for every fixed nonempty (possibly infinite) $B \subseteq X$.
% %\end{enumerate}
% 
% %\noindent A similar result holds if $(X,\nu)$ is irreducible.
 \end{enumerate}
\end{teo}
As a consequence we have the following corollary.

\begin{cor}\label{cor:pureweak-nonstrong-discrete}
 Consider two BRWs  $(X,\mu)$ and $(X,\nu)$. Suppose that  $A \subseteq X$ is a nonempty set
 such that $\mu_x=\nu_x$ for all $x \not \in A$. 
 \begin{enumerate}
  \item Suppose that $(X,\mu)$ dies out locally in $A$ from all $x \in X$
 and $(X,\nu)$ survives globally from all $x \in X$; then
 \[
%   \textit{global extinction for }  (X,\mu) \textrm{ at }A\textrm{ from all }x \in X
  \bar{\mathbf{q}}^\mu(x)=1\textrm{ for all }x \in X
  \Longleftrightarrow
   \textrm{strong local survival for } (X,\nu) \textrm{ at }A\textrm{ from all }x\in X.
  \]
  \item If $(X,\mu)$ dies out globally from all $x \in X$
 and $(X,\nu)$ survives globally from all $x \in X$ then there is strong local survival for $(X,\nu)$ in $A$
 from all $x \in X$.
  \end{enumerate}
\end{cor}
The following corollary describes how a small and local modification can affect the phase diagram of a continuous-time BRW.
\begin{cor}\label{cor:pureweak-nonstrong}
 Let $(X,K)$ and $(X,K^\prime)$ two irreducible continuous-time BRWs such that $k_{xy}=k^\prime_{xy}$ for all $x \in X \setminus A$
 where $A$ is a nonempty, finite set. 
 %Suppose that $\lambda_w < \lambda_s$. 
 Then the following are equivalent:
 \begin{enumerate}
  \item $\lambda^\prime_w<\lambda_w$;
  \item $\lambda^\prime_s<\lambda_w$;
\item $\lambda^\prime_w=\lambda^\prime_s<\lambda_w$.
 \end{enumerate}
 Moreover if one of the previous holds, for the BRW $(X,K^\prime)$
 \begin{enumerate}[(i)]
  \item if $\lambda \le \lambda_w^\prime$ there is a.s.~global and local extinction in every nonempty set $B$;
     \item if $\lambda \in (\lambda_w^\prime, \lambda_w)$ there is strong local survival in every nonempty set $B$;
     \item if $\lambda =\lambda_w$ and the $(X,K)$-BRW dies out globally, then there is strong local survival in $B$
     for every nonempty set $B$, otherwise there is non-strong local survival in $B$ for every nonempty finite set $B$;
   \item if $\lambda \in (\lambda_w, \lambda_s]$ (when non empty) there is non-strong local survival in every nonempty finite set $B$;
    \item if $\lambda > \lambda_s$ then local survival is strong (resp.~non-strong) in a nonempty finite set $B$ if and only if 
  the same holds for $(X,K)$.  
 \end{enumerate}
 
\end{cor}
We already pointed out that at $\lambda=\lambda_w$, global survival is possible. This cannot 
happen if the process is a finite modification of another BRW, as in Corollary~\ref{cor:pureweak-nonstrong}.
An easy way to modify a BRW $(X,K)$ in order to obtain  $\lambda^\prime_s<\lambda_w$, is to add 
a sufficiently rapid reproduction from $y$ to $y$ (for a fixed $y$).

We now apply Corollary~\ref{cor:pureweak-nonstrong} to the following example (see also \cite[Example 4.2]{cf:BZ14-SLS})
which can be discussed without using cumbersome arguments such as those contained in \cite[Remark 3.2]{cf:BZ}
and \cite[Example 4.1]{cf:BZ14-SLS}.

\begin{exmp}\label{th:nonstrongandstrong}
Consider the edge-breeding continuous-time BRW on the homogeneous
tree $\mathbb{T}_d$ with degree $d\ge 3$; in this case $K$ is the adjacency matrix. 
It is easy to prove (see for instance \cite[Example 4.2]{cf:BZ14-SLS}) that 
$\lambda_w=1/d< 1/2\sqrt{d-1}=\lambda_s$.
If $\lambda \le \lambda_w$ there is global extinction,
if $\lambda > 1/2\sqrt{d-1}$ there is strong local survival (see \cite[Corollary 3.2]{cf:BZ14-SLS})
while 
if $\lambda \in (1/d ,1/2 \sqrt{d-1}]$ the probability
of global survival is positive and
independent of the starting point and the probability of local survival in any finite $A \subseteq X$
is $0$. The phase diagram is shown by Figure~\ref{fig:t3}.

Fix a vertex $y \in \mathbb{T}_d$ and denote by $A$ the singleton $\{y\}$. Let us modify the BRW
by adding a loop at $y$, that is, by considering a new matrix $K^\prime$ where all the entries are the same as 
those of $K$ but $k^\prime_{yy}>d$.
Hence $\lambda^\prime_s\le1/k^\prime_{yy}<1/d=\lambda_w$ and Corollary~\ref{cor:pureweak-nonstrong} applies.
As a result, $\lambda^\prime_s=\lambda^\prime_w$ and we have the following behaviour for
$(\mathbb{T}_d, K^\prime)$ (see Figure~\ref{fig:t3modified}):
if  $\lambda < \lambda_w^\prime$ there is global extinction,
if $\lambda \in (\lambda_w^\prime, 1/d]$ there is strong local survival,
if $\lambda \in (1/d, 1/2\sqrt{d-1}]$ there is non-strong local survival
and if $\lambda > 1/2\sqrt{d-1}$ there is strong local survival again.

 \begin{figure}[H]
 \centering
 \begin{minipage}{0.45\textwidth}
 \centering
  \includegraphics[width=0.95\textwidth]{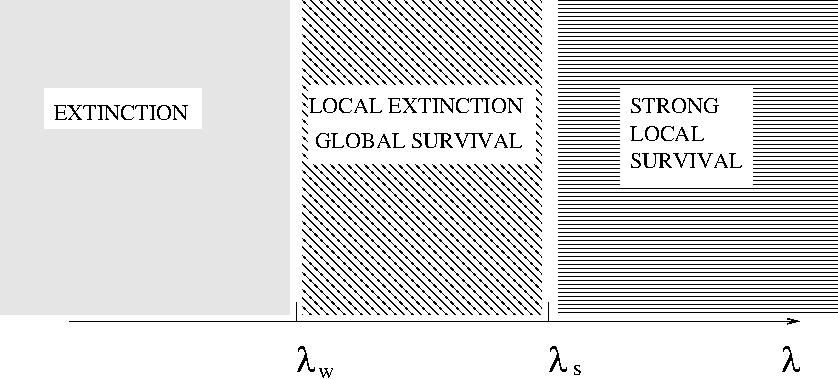}
  \caption{\scriptsize{Phase diagram for $(\mathbb{T}_d, K)$.}}\label{fig:t3}
 \end{minipage}
\begin{minipage}{0.45\textwidth}
\centering
 \includegraphics[width=0.95\textwidth]{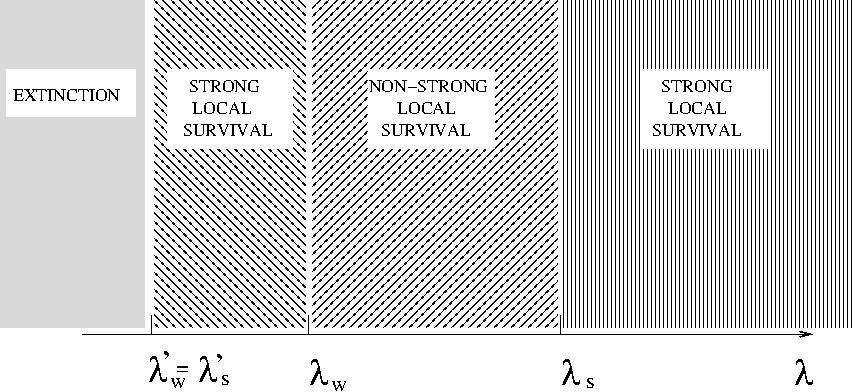}
  \caption{\scriptsize{Phase diagram for $(\mathbb{T}_d, K^\prime)$.}}\label{fig:t3modified}
\end{minipage}
\end{figure}

We note that, as it always happens in a continuous-time BRW, $\mathbf{q}(\cdot,A)$ depends on $\lambda$
since a continuous-time BRW actually is a family of processes indexed by $\lambda$.
The function $\lambda \mapsto \mathbf{q}(\cdot,A)$ does not need to be continuous.
Consider, for instance, the above edge-breeding BRW $(\mathbb{T}_d, K)$; if we
look for the global extinction probability vector it is easy to show, by using equation~\eqref{eq:Gcontinuous} and the
equality $\bar {\mathbf{q}}=\lim_{n \to \infty} G^{(n)}(\mathbf{0})$, that $\bar {\mathbf{q}}(x)=\min(1, (d\lambda)^{-1})$
which is a nice continuous function. On the other hand, if we consider ${\mathbf{q}}(x,x)$ (where $x \in \mathbb{T}_d$)
then it equals $\mathbf{1}$ in the interval $(0, 1/2\sqrt{d-1}]$ and $(d\lambda)^{-1} \cdot \mathbf{1}$ in the interval $(1/2\sqrt{d-1},
+\infty)$; thus there is a discontinuity at $1/2\sqrt{d-1}$.
\end{exmp}

\section{Proofs}
\label{sec:proofs}

 \begin{proof}[Details on Remark~\ref{rem:irreducible}]
 If the BRW is irreducible we have ${\mathbf{q}}(v,h)= {\mathbf{q}}(v,v)$ for all $v,h \in X$, which implies
${\mathbf{q}}(v,A)={\mathbf{q}}(v,v)={\mathbf{q}}(v,B)$ for all $v \in X$ and finite, nonempty sets $A$ and $B$. Indeed
if the process visits infinitely many times $A$ starting
from $v$ then
it visits infinitely many times at least a vertex $h \in A$ and, by irreducibility, it visits infinitely many times 
$v$. Similarly, if the process visits infinitely many times $v$ starting from $v$ then it visits infinitely many times
any vertex $h \in A$.

If $\bar {\mathbf{q}}(x)=1$ then ${\mathbf{q}}(y,B)=1$ for all $y \in X$ and $B \subseteq X$ and there is nothing to prove.
 Suppose that $\bar {\mathbf{q}}(x)={\mathbf{q}}(x,A)<1$ and, by contradiction, $\bar {\mathbf{q}}(y)<{\mathbf{q}}(y,B)$ 
 for some $x,y \in X$ and $A,B \subseteq X$
finite. 
We know that there is a positive probability that the process, starting from $x$
has at least one descendant at $y$. There is also a positive probability that all the particles (except one at $y$)
die and the progeny of the surviving particle survives globally but not locally in $A$. Thus,
there is a positive probability, starting from $x$, of surviving globally but not locally in $A$ and this is a contradiction.
Hence $\bar {\mathbf{q}}(y) = {\mathbf{q}}(y,A)$ for all $y \in X$. But we proved above that, in the irreducible case,  
${\mathbf{q}}(v,A)={\mathbf{q}}(v,B)$ for all $v \in X$ and all finite nonempty subsets $A$ and $B$, whence
$\bar {\mathbf{q}}(y) = {\mathbf{q}}(y,A)$ for all $y \in X$ and every finite nonempty subset  $B$.
If $B$ is infinite and $z \in B$ then $\bar {\mathbf{q}}(y)={\mathbf{q}}(y,z) \ge {\mathbf{q}}(y,B) \ge \bar {\mathbf{q}}(y)$
for all $y \in X$.
\end{proof}

\begin{proof}[Proof of Theorem~\ref{th:strongconditioned}]
 The equivalence between $(1)$, $(2)$, $(3)$ and $(4)$ was already proven in \cite[Theorem 3.3]{cf:BZ14-SLS}. Clearly,
 $\pr^x(A_0 \setminus GE)=0$ implies ${\mathbf{q}}_0(x,A)\le \bar{\mathbf{q}}(x)$, hence $(5) \Longrightarrow (2)$. We prove now
 that $(1) \Longrightarrow (5)$. Indeed, define $A_n:=$``visit $A$ at most $n$ times''. Hence, $A_{n+1} \supseteq
 A_n$ and $\bigcup_{n \in \mathbb{N}} A_n \supseteq GE$. Note that ${\mathbf{q}}(x,A)=\pr^x \big ( \bigcup_{n \in \mathbb{N}} A_n \big )$
 and $\bar{\mathbf{q}}(x)=\pr^x(GE)$. If ${\mathbf{q}}(x,A)= \bar{\mathbf{q}}(x)$ then 
 $\pr^x \big ( \bigcup_{n \in \mathbb{N}} A_n \setminus GE\big )=0$ which is equivalent to 
 $\pr^x ( A_0 \setminus GE\big )=0$ for all $n \in \N$.
\end{proof}

\begin{proof}[Proof of Theorem~\ref{th:modifiedBRW}]
 \begin{enumerate}
% %  \item 
%  From Theorem~\ref{th:strongconditioned} we have that there exists $y \in X$ such that $q_0(y,A)> \bar q(x)$. This is equivalent to 
 \item
We note that $(X,\mu)$ and $(X,\nu)$ have the same behaviour until they first hit $A$,
hence ${\mathbf{q}}_0^\mu(x,A)={\mathbf{q}}_0^\nu(x,A)$ for all $x \not \in A$. If $x \in A$ then clearly
${\mathbf{q}}_0^\mu(x,A)=0={\mathbf{q}}_0^\nu(x,A)$.

Suppose now that ${\mathbf{q}}^\mu(\cdot,A) \not = \bar {\mathbf{q}}^\mu(\cdot)$. Hence, according to Theorem~\ref{th:strongconditioned}
(see comments after its statement),
there exists  $x \in X\setminus A$ such that there is a positive probability of survival starting from $x$ without ever visiting $A$.
Since the two processes have the same behaviour until they first hit $A$, the same holds for $(X,\nu)$ and this implies that 
${\mathbf{q}}^\nu(x,A)>\bar {\mathbf{q}}^\nu(x)$; thus
${\mathbf{q}}^\nu(\cdot,A) \not =\bar {\mathbf{q}}^\nu(\cdot)$.
% \end{enumerate}

\item
By Remark~\ref{rem:irreducible}, when $B$ and $C$ are finite nonempty subsets, 
${\mathbf{q}}^\mu(\cdot,B)={\mathbf{q}}^\mu(\cdot,C)$, whence the implication is trivial.

Moreover, recall that $B \subseteq C$ implies ${\mathbf{q}}^\mu(\cdot,B) \ge {\mathbf{q}}^\mu(\cdot,C)$.
hence, if $C$ is infinite and $z \in C$ then, for all $y \in X$, Remark~\ref{rem:irreducible} yields 
\[
\bar {\mathbf{q}}^\mu(y)={\mathbf{q}}^\mu(y,B)={\mathbf{q}}^\mu(y,z) \ge {\mathbf{q}}^\mu(y,C) \ge \bar {\mathbf{q}}^\mu(y).
\]
\end{enumerate}
\end{proof}

\begin{proof}[Proof of Corollary~\ref{cor:pureweak-nonstrong-discrete}]
\begin{enumerate}
 \item
 According to the hypotheses $\mathbf{q}^\mu(\cdot,A)=\mathbf{1}>\bar{\mathbf{q}}^\nu(\cdot)$.
 Hence if $\bar{\mathbf{q}}^\mu(\cdot)=\mathbf{1}=\mathbf{q}^\mu(\cdot,A)$ then, according to
 Theorem~\ref{th:modifiedBRW}(1), $\mathbf{q}^\nu(\cdot,A)=\bar{\mathbf{q}}^\nu(\cdot)<\mathbf{1}$,
 that is, there is strong local survival for $(X,\nu)$ in $A$ from every $x \in X$.
 Conversely, $\mathbf{q}^\nu(\cdot,A)=\bar{\mathbf{q}}^\nu(\cdot)<\mathbf{1}$ implies, by Theorem~\ref{th:modifiedBRW}(1),
 $\bar{\mathbf{q}}^\mu(\cdot)=\mathbf{q}^\mu(\cdot,A)=\mathbf{1}$, thus global extinction from every $x \in X$.
\item If $(X,\mu)$ dies out globally from all $x \in X$ then it dies out locally in $A$ from all $x \in X$ hence, from the previous
part, there is  strong local survival for $(X,\nu)$ in $A$ from every $x \in X$.
\end{enumerate}
 \end{proof}

\begin{proof}[Proof of Corollary~\ref{cor:pureweak-nonstrong}]
Observe that the discrete-time counterparts of these continuous-time BRWs satisfy the hypotheses
of Theorem~\ref{th:modifiedBRW}, namely, their offspring distribution are the same outside $A$.
 
 Clearly %$(3) \Longrightarrow (1)$, 
 $(2) \Longrightarrow (1)$ and $(3) \Longrightarrow (2)$.
 We just need to prove that $(1) \Longrightarrow (3)$; more precisely, we prove that 
 $\lambda^\prime_w<\lambda_w \Longrightarrow \lambda^\prime_w=\lambda^\prime_s$. Take $\lambda \in (\lambda_w^\prime, \lambda_w)$;
 the $\lambda$-$(X,K^\prime)$ BRW survives globally, hence $\bar{\mathbf{q}}^\prime < \mathbf{1}$.
 On the other hand,  $\mathbf{1}=\bar{\mathbf{q}}=\mathbf{q}(\cdot,A)$ whence, according
 to Theorem~\ref{th:modifiedBRW}(1), $\bar{\mathbf{q}}^\prime=\mathbf{q}^\prime(\cdot,A)$ which implies
 $\mathbf{q}^\prime(\cdot,A)< \mathbf{1}$. If the $\lambda$-$(X,K^\prime)$ BRW survives locally in
 the finite set $A$ it means that it survives locally at a vertex $x \in A$ ($\Longleftrightarrow$ at every vertex, since
 the process is irreducible). This implies $\lambda \ge \lambda_s^\prime$; thus $\lambda_s^\prime=\lambda_w^\prime$.
 
 Note that in the discrete-time counterpart of a continuous-time BRW 
 every particle at every vertex has a positive probability of dying without
 breeding; hence by Remark~\ref{rem:irreducible} strong local survival is a common property of all
 starting vertices.
 
 We consider the following disjoint intervals for $\lambda$.
 \begin{enumerate}[(i)]
  \item Suppose that $\lambda < \lambda_w^\prime$; by definition there is global, hence local, extinction.
  If $\lambda = \lambda_w^\prime$ then, 
  according to \cite[Theorem 4.7]{cf:BZ2} (see also 
  \cite[Theorem 3.5 and Section 4.2]{cf:BZ}), since $\lambda=\lambda_w^\prime=\lambda_s^\prime$ then
  the $\lambda$-$(X,K^\prime)$ BRW dies out locally (at any finite set $C$),
  hence $\mathbf{q}^\prime(\cdot,C)= \mathbf{1}$ (clearly, being $\lambda < \lambda_w$, 
  $\mathbf{q}(\cdot,B)=\bar{\mathbf{q}}= \mathbf{1}$ for all $B\subseteq X$), using Theorem~\ref{th:modifiedBRW}(1),
  \[
   \mathbf{q}(\cdot,A)=\bar{\mathbf{q}} \Longrightarrow \bar{\mathbf{q}}=\mathbf{q}^\prime(\cdot,A)=\mathbf{1}.
  \]
  Since $\bar{\mathbf{q}} \le \mathbf{q}^\prime(\cdot,B)$ for all $B$, we have $\mathbf{q}^\prime(\cdot,B)=\mathbf{1}$.
 \item $\lambda \in (\lambda_w^\prime, \lambda_w)$. By definition, since $\lambda_w^\prime=\lambda_s^\prime$,
 there is global and local survival for the $\lambda$-$(X,K^\prime)$ BRW. This implies that 
 $\bar{\mathbf{q}}^\prime \le \mathbf{q}^\prime(\cdot,B) < \mathbf{1}$ for every set $B$. On the other hand,
 there is global and local extinction for the $\lambda$-$(X,K)$ BRW which implies
 $\bar{\mathbf{q}}= \mathbf{q}(\cdot,B)= \mathbf{1}$. Again, according to Theorem~\ref{th:modifiedBRW}(1),
 $\bar{\mathbf{q}}^\prime = \mathbf{q}^\prime(\cdot,A) < \mathbf{1}$, that is, strong local survival in $A$.
 Irreducibility implies $\bar{\mathbf{q}}^\prime = \mathbf{q}^\prime(\cdot,B)$ for every (finite or infinite) set $B$.
 \item Clearly, since $\lambda =\lambda_w \le \lambda_s$, we have $\mathbf{q}(\cdot,B)= \mathbf{1}$ for all finite subsets $B$.
  Hence
  \[
   \lambda-(X,K) \textrm{ survives globally} \Longleftrightarrow \bar{\mathbf{q}}< \mathbf{q}(\cdot,A)
  \]
that is, according to Theorem~\ref{th:modifiedBRW}(1), if and only if $\bar{\mathbf{q}}^\prime < \mathbf{q}^\prime(\cdot,A)$.
This, again, implies $\bar{\mathbf{q}}^\prime < \mathbf{q}^\prime(\cdot,B)$ for every nonempty finite subset $B$.
If, on the other hand, $\lambda-(X,K)$  dies out globally, then $\bar{\mathbf{q}}^\prime = \mathbf{q}^\prime(\cdot,A)$ and
  $\bar{\mathbf{q}}^\prime = \mathbf{q}^\prime(\cdot,B)$ for every nonempty subset $B$.
  \item $\lambda \in (\lambda_w, \lambda_s]$ (we suppose that the interval is nonempty, otherwise there is nothing to prove). 
  Here we have $\bar{\mathbf{q}}< \mathbf{1}=\mathbf{q}(\cdot,B)$ for every finite subset $B$.
  Theorem~\ref{th:modifiedBRW}(1) yields $\bar{\mathbf{q}}^\prime < \mathbf{q}^\prime(\cdot,A)< \mathbf{1}$  and, by irreducibility,
  $\bar{\mathbf{q}}^\prime < \mathbf{q}^\prime(\cdot,B)< \mathbf{1}$ for every finite, nonempty subset $B$.
  \item $\lambda > \lambda_s$. Now, $\mathbf{q}(\cdot,B)< \mathbf{1}$ and $\mathbf{q}^\prime(\cdot,B)< \mathbf{1}$
  for every nonempty $B \subset X$.
  Again, by Theorem~\ref{th:modifiedBRW}(1), we have
  \[
 {\mathbf{q}}^\mu(\cdot,A)=\bar {\mathbf{q}}^\mu(\cdot) \Longleftrightarrow
 {\mathbf{q}}^\nu(\cdot,A)=\bar {\mathbf{q}}^\nu(\cdot).
\]
  If $B$ is finite then Theorem~\ref{th:modifiedBRW}(2) yields the conclusion.
 \end{enumerate}
\end{proof}

\begin{proof}[Details on Example~\ref{exmp:irreduciblemorefixedpoints}]
We are considering the BRW on $\N$ where every particle at $n$ has two children at $n+1$ with probability $p-\varepsilon$,
one child at $\max(0,n-1)$ with probability $\varepsilon$
and no children with probability $1-p$. We fixed $p<1/\sqrt{2}$ and $\varepsilon(p-\varepsilon) \le 1/8$.
We know that $M_w=2p-\varepsilon>1$ and now we compute $M_s$.
% (where $4/3-\varepsilon>1$ in order to have a global supercritical BRW);
% $\varepsilon$ will be a small positive constant (clearly smaller than $2/3$) which will be determined later on.
% Note that the same holds if, instead of $2/3$, we take a generic $p \in (1/2, 1]$ such that $2p-\varepsilon>1$.
% 
% As in the previous example, this is an irreducible BP-like BRW;
% according to Theorem~\ref{th:discretesurv}, global and local survival depends on $M_w$ and $M_s$.
% To compute these parameters we refer to (\cite{cf:BZ, cf:BZ2} and \cite[Section 4.6]{cf:BZ13}).
% In particular  is the expected number of children $4/3-\varepsilon$ that we assumed to be strictly larger than one.
%  Moreover we have 
%  (see \cite[Proposition 4.33]{cf:BZ13})  
% $G(\mathbf{z}|n)=(2/3-\varepsilon) \mathbf{z}(n+1)^2 +\mathbf{z}(\max(0,n-1))\varepsilon+1/3$ and 
% $\bar{\mathbf{q}}=(2-3\varepsilon)^{-1} \cdot \mathbf{1}$. 
More precisely, we prove that, %, for every sufficiently small $\varepsilon>0$,
given $\varepsilon(p-\varepsilon) \le 1/8$, we have local extinction, that is, $M_s\le 1$.
Indeed, $1/M_s=\max\{z \ge 0\colon \Phi(x,x|z) \le 1\}$ where %is the convergence radius of the power series 
$\Phi(x,y|z):=\sum_{n=1}^\infty\phi^n(x,y)z^n$ and $\phi^n(x,y)$
is the expected progeny at $y$ of a particle which is at $x$ at time 0, along an $n$-step 
reproduction trail which hits $y$ for the first time at step $n$ (see \cite[Section~2.2]{cf:BZ2}
and \cite[Sec.~3.2]{cf:Z1}). It is easy to see that $\Phi(0,0|z)=\varepsilon z+2(p-\varepsilon)z\Phi(1,0|z)$,
$\Phi(1,0|z)=\varepsilon z+{2}(p-\varepsilon)z\Phi(2,0|z)$
and $\Phi(2,0|z)=(\Phi(1,0|z))^2$. Solving the quadratic equation in $\Phi(1,0|z)$ and choosing the solution
which has a finite limit as $z\to0$, we get that
\[ \begin{split}
 \Phi(1,0|z)&=\frac{1-\sqrt{1-{8}\varepsilon z^2(p-\varepsilon)}}{{4}z(p-\varepsilon)},\\
 \Phi(0,0|z)&=\varepsilon z+\frac{1-\sqrt{1-{8}\varepsilon z^2(p-\varepsilon)}}{{2}}.\\
 \end{split}
\]
Clearly $M_s \le 1$ if and only if $\Phi(x,x|1) \le 1$ which, in turn, is equivalent to $8\varepsilon(p-\varepsilon) \le 1$ and 
$2 \varepsilon -1\le \sqrt{1-8 \varepsilon (p-\varepsilon)}$. Note that $2p-\varepsilon >1$ and $p<1/\sqrt{2}$, hence
$2\varepsilon-1 < 4p-3<2\sqrt{2}-3<0$; thus $\Phi(x,x|1) < 1$ and $M_s \le 1$.

% Hence %$M_s=2 \sqrt{2\varepsilon(p-\varepsilon)}\le p\sqrt{2}$
% $M_s=2\sqrt{{2}\varepsilon(p-\varepsilon)}\le 1$ 
% provided that $\Phi\big(0,0|1/\sqrt{8\varepsilon(p-\varepsilon)}\big)\le 1$. This is equivalent to $2p-3 \varepsilon \ge 0$;
% moreover, since $2p-\varepsilon >1$ and $p \le 1/\sqrt{2}$, we have $2p-3\varepsilon > 3 -4p > 3-4/\sqrt{2}>0$.
% Thus $M_s \le 1$ since we assumed $\varepsilon(p-\varepsilon) \le 1/8$.

Let us compute the set of fixed points; we prove %that the condition $\varepsilon \le 2/9$ implies 
the existence of an uncountable number
of fixed points. %Note that $\varepsilon \le 2/9$ implies that the BRW is globally supercritical.
Clearly if 
$\mathbf{z}(0) =(1-p)/(p-\varepsilon)$ (resp.~$\mathbf{z}(0) =1$) we have $\mathbf{z}=(1-p)/(p-\varepsilon) \cdot \mathbf{1}$ 
(resp.~$\mathbf{z}=\mathbf{1}$).
This gives the two constant fixed points (the smallest one $\bar{\mathbf{q}}$ and the largest one): 
observe that these constants are the solutions
of $J(x)=0$, where $J(x):= (p-\varepsilon) x^2-(1-\varepsilon)x+1-p$. Hence $J(x)<0$ for all 
$x \in \big ((1-p)/(p-\varepsilon),1 \big)$.
Any other fixed point must satisfy $\bar{\mathbf{q}} < \mathbf{z} < \mathbf{1}$, thus 
$\mathbf{z}(0) \in \big ((1-p)/(p-\varepsilon),1 \big )$.
We prove by induction that, whenever we fix $\mathbf{z}(0) \in \big ((1-p)/(p-\varepsilon),1 \big )$, then
\[
 (P_n)= \ \begin{cases}
          \mathbf{z}(n)>\mathbf{z}(n-1)\\
          \mathbf{z}(n) \in \big ((1-p)/(p-\varepsilon),1 \big )\\
          1-\mathbf{z}(n) > \frac{1-\mathbf{z}(n-1)}{2p}
         \end{cases}
\]
hold for every $n \ge 1$. 
This will prove that any suitable choice of $\mathbf{z}(0)$ gives a fixed point.
The previous conditions are clearly redundant but it is easier to proceed like this.
Using the equation, $G(\mathbf{z})=\mathbf{z}$, we have
\begin{equation}\label{eq:1}
 \begin{cases}
  \mathbf{z}(1)=\sqrt{\frac{(1-\varepsilon)\mathbf{z}(0)-(1-p)}{p-\varepsilon}}& \\
  \mathbf{z}(n+1)=\sqrt{\frac{\mathbf{z}(n)-\varepsilon\mathbf{z}(n-1)-(1-p)}{p-\varepsilon}}& \textrm{if }n \ge 1.\\
 \end{cases}
\end{equation}
Since $\mathbf{z}(0) \in \big ((1-p)/(p-\varepsilon),1 \big )$ and 
$\mathbf{z}(1)^2-\mathbf{z}(0)^2=-J(\mathbf{z}(0))/(p-\varepsilon) >0$ then
$\mathbf{z}(1)>\mathbf{z}(0)>(1-p)/(p-\varepsilon)$. 
Clearly $\mathbf{z}(1)=\sqrt{(\big (1-\varepsilon)\mathbf{z}(0)-(1-p \big ))/(p-\varepsilon)}<
\sqrt{\big (1-\varepsilon-(1-p)\big )/(p-\varepsilon})<1$.
Moreover, using the previous inequality, 
\[
 \begin{split}
  1-\mathbf{z}(1)&=\frac{1-\frac{(1-\varepsilon)\mathbf{z}(0)-(1-p)}{p-\varepsilon}}
  {1+\sqrt{\frac{(1-\varepsilon)\mathbf{z}(0)-(1-p)}{p-\varepsilon}}}
  = \frac{1-\frac{(1-\varepsilon)\mathbf{z}(0)-(1-p)}{p-\varepsilon}}
  {1+\mathbf{z}(1)}
  > \frac{1-\mathbf{z}(0)}{2} \cdot \frac{1-\varepsilon}{p-\varepsilon}
  > \frac{1-\mathbf{z}(0)}{2p}
 \end{split}
\]
thus $(P_1)$ holds.

Let us prove that $(P_n) \Longrightarrow (P_{n+1})$. Using $\mathbf{z}(n)>\mathbf{z}(n-1)$ we have
$\mathbf{z}(n+1)^2-\mathbf{z}(n)^2>(\mathbf{z}(n)-\varepsilon\mathbf{z}(n)-(1-p))/(p-\varepsilon)-\mathbf{z}(n)^2=
-J(\mathbf{z}(n))/(p-\varepsilon) >0$ where the last inequality comes from $\mathbf{z}(n) \in \big ((1-p)/(p-\varepsilon),1 \big )$.
Hence $\mathbf{z}(n+1)>\mathbf{z}(n)>(1-p)/(p-\varepsilon)$. On the other hand, using $\mathbf{z}(n)<1$
and $1-\mathbf{z}(n-1) \le 2p(1-\mathbf{z}(n))$, 
\[
 \begin{split}
  1-\mathbf{z}(n+1)&=\frac{1-\frac{\mathbf{z}(n)-\varepsilon\mathbf{z}(n-1)-(1-p)}{p-\varepsilon}}
  {1+\sqrt{\frac{\mathbf{z}(n)-\varepsilon\mathbf{z}(n-1)-(1-p)}{p-\varepsilon}}}
= \frac{1-\mathbf{z}(n)-\varepsilon(1-\mathbf{z}(n-1))}
  {(p-\varepsilon)(1+\mathbf{z}(n+1))}\\
  &>  (1-\mathbf{z}(n))\frac{1-2p \varepsilon}{(p-\varepsilon)(1+\mathbf{z}(n+1))}=(\$)\\
   \end{split}
\]
which implies $\mathbf{z}(n+1)<1$ (since $p-\varepsilon >1-p>0$ whence $1-2p\varepsilon >1-2p^2>0$ whenever 
$p<1/\sqrt2$). Using this last inequality
(and the bound  $p<1/\sqrt2$), we prove the last part of $(P_{n+1})$:
\[
 (\$) > (1-\mathbf{z}(n))\frac{1-2p \varepsilon}{2(p-\varepsilon)}=\frac{1-\mathbf{z}(n)}{2} \Big (
 \frac{\varepsilon(1-2p^2)}{p(p-\varepsilon)}+ \frac{1}{p}
 \Big )>\frac{1-\mathbf{z}(n)}{2p}.
\]
% Now that we proved that $\matbnf{z}(n) <1$ we can improve $(1-\matbnf{z}(n+1))/(1-\matbnf{z}(n))$
% as in the case $n=0$ by putting $1+\matbnf{z}(n+1)$ in the denominator (instead of $1+\sqrt{\ldots}$
% and using the fact that $1+\matbnf{z}(n+1)<2$. The conclusio is easily $(1-\matbnf{z}(n+1))/(1-\matbnf{z}(n))> 3/4$
% if $\varepsilon \le 2/9$.
Hence the set of fixed points is uncountable.
\end{proof}

\section*{Acknowledgements}

The authors are grateful to the organizers of the XIX Brazilian Summer School in Probability for the invitation
and financial support. They also acknowledge financial support from INDAM (Istituto Nazionale di Alta Matematica).

\end{document}